\documentclass[12pt]{article}
\usepackage{amsfonts, amsmath, amsthm, amssymb, relsize, scalerel, graphics, tabu}
\usepackage{bbm}
\usepackage[table]{xcolor}
\usepackage [english]{babel}
\usepackage{cjhebrew}
\usepackage{mathrsfs}
\usepackage{enumitem}
\usepackage [autostyle, english = american]{csquotes}
\usepackage{harpoon}

\newtheorem{theorem}{Theorem}[section]
\newtheorem{proposition}[theorem]{Proposition}
\newtheorem{corollary}[theorem]{Corollary}
\newtheorem{lemma}[theorem]{Lemma}

\theoremstyle{definition}

\newtheorem{definition}[theorem]{Definition}

\newcommand{\ii}{\mathrm{i}}

\title{Weighted $L^p$ Estimates for the Bergman and Szeg\H{o} Projections on Strongly Pseudoconvex Domains with Near Minimal Smoothness}
\author{
Nathan A. Wagner\footnote{Supported by NSF GRF, grant number DGE-1745038}\\
Department of Mathematics and Statistics\\
Washington University in St. Louis\\
{\tt nathanawagner@wustl.edu}\vspace{3mm}\\
Brett D. Wick \footnote{Supported by NSF grants DMS-1800057 and DMS-1560955, as well as ARC DP190100970.}\\
Department of Mathematics and Statistics\\
Washington University in St. Louis\\
{\tt wick@math.wustl.edu}\\
}

\begin{document}

\maketitle

\begin{abstract} We prove the weighted $L^p$ regularity of the ordinary Bergman and Cauchy-Szeg\H{o} projections on strongly pseudoconvex domains $D$ in $\mathbb{C}^n$ with near minimal smoothness for appropriate generalizations of the $B_p/A_p$ classes. In particular, the $B_p/A_p$ Muckenhoupt type condition is expressed relative to balls in a quasi-metric that arises as a space of homogeneous type on either the interior or the boundary of the domain $D$.

\end{abstract}

\vspace{10pt}

\section{Introduction}\label{theproblem}
\subsection{The Problem}
Let $D$ be a bounded, strongly pseudoconvex domain in $\mathbb{C}^n$ with sufficiently smooth (say $C^k$, with $k \geq 2$) boundary, and let $bD$ denote the boundary of $D$. Thus, there is a real-valued, $C^k$, and strictly plurisubharmonic defining function $\rho$ such that $D=\{z: \rho(z)<0\}$ and $\nabla\rho \neq 0$ on $bD$. Recall that the Bergman projection is the orthogonal projection $\mathcal{B}: L^2(D) \rightarrow A^2(D)$, where $A^2(D)$ is the Bergman space, which is the space consisting of those holomorphic functions on $D$ which are square-integrable. Similarly, we recall that the Cauchy-Szeg\H{o} (or simply Szeg\H{o}) projection $\mathcal{S}: L^2(bD)\rightarrow H^2(bD)$ is the orthogonal projection of $L^2(bD)$ onto the holomorphic Hardy space $H^2(bD)$. We define $H^2(bD)$ to be the following closure in $L^2(bD)$:
$$H^2(bD):= \overline{\{f \in L^2(bD): f=F \lvert_{bD}, F \in \text{Hol}(D) \hspace{0.2 cm} \text{and} \hspace{0.2 cm} F \in C^0(\overline{D})\}}.$$

It is of interest to determine when the Bergman and Szeg\H{o} projections extend to bounded operators on $L^p.$ In the simpler case that $D$ has smooth ($C^\infty$) boundary, it has been been known for decades that $\mathcal{B}$ and $\mathcal{S}$ extend to bounded operators  on $L^p$ if $1<p<\infty$ because direct estimates can be obtained on the Bergman and Szeg\H{o} kernels in both cases (see \cite{PS}). If the domain is less regular, a more indirect approach is needed because it is hard to obtain direct estimates on the kernels. Kerzman and Stein around the same time developed a powerful idea that allowed them to relate the Szeg\H{o} projection $\mathcal{S}$ to a ``Cauchy" integral operator $\mathcal{C}$ via an operator equation (see \cite{KS1,KS2} for the one variable and several variable cases, respectively). The essential idea, exploited in \cite{KS1,KS2} as well as numerous other papers in the literature, involves constructing  an auxiliary operator $\mathbf{C}$ that also produces and reproduces holomorphic functions inside $D$ from boundary data, and defining $\mathcal{C}$ to be a restriction of $\mathbf{C}$ to the boundary in an appropriate sense, so that $\mathcal{C}$ is a singular integral operator. This operator $\mathbf{C}$ is given as a sum, $\mathbf{C}_1+\mathbf{C}_2$, where $\mathbf{C}_1$ is constructed using the theory of Cauchy-Fantappi\'e integrals and $\mathbf{C}_2$ is a correction term obtained by solving a $\overline{\partial}$ problem on a strongly pseudoconvex, smoothly bounded domain that contains $\overline{D}$ (see, for example, \cite{KS2,LS1,R}). Importantly, $\mathbf{C}_1$ has a completely explicit kernel. The operator $\mathcal{C}^*-\mathcal{C}$ then roughly measures the ``error" introduced by considering $\mathcal{C}$ instead of $\mathcal{S}.$\footnote{Throughout this paper, we use the symbol $*$ to denote the adjoint of an operator on $L^2(bD)$. Importantly, the adjoint is taken on the \emph{unweighted} Lebesgue space.} A similar trick can be employed for the Bergman projection.
 
Since (Levi) pseudoconvexity is formulated in terms of second derivatives, if we are to restrict our attention to strongly pseudoconvex domains (which are the domains on which the above auxiliary operators can be constructed), we must at minimum assume that the boundary of our domain $D$ is $C^2$ for these questions to make sense. By passing through these auxiliary operators, which have ``non-canonical" kernels and are constructed using the theory of holomorphic integral representations, it is possible to obtain boundedness properties for the operators $\mathcal{B}$ and $\mathcal{S}$ without relying on explicit bounds for the Bergman or Szeg\H{o} kernels. This indirect approach was employed by Lanzani and Stein in \cite{LS1,LS2} to study these problems in the case that $D$ has $C^2$ boundary. These results are the best possible on strongly pseudoconvex domains. 

In harmonic analysis, it is very common to consider the boundedness of integral operators on weighted spaces. The consideration of these problems goes back to the formulation of the $A^p$ condition for the Hilbert transform by Hunt, Muckenhoupt, and Wheeden, see \cite{HMW}.  In the context of the Szeg\H{o} and Bergman projections, there seems to be two distinct questions that one could ask. One could consider the \emph{weighted} Szeg\H{o} (respectively Bergman) projection, which is the projection from $L^2_\sigma(bD)$ to $H^2_\sigma(bD)$, where $\sigma$ is a weight, and try to determine for which weights this projection is bounded on $L^p(bD)$ or $L^p_\sigma(bD)$. Alternately, one could consider the \emph{ordinary} Szeg\H{o} (Bergman) projection acting as an operator on weighted spaces $L^p_\sigma(bD)$. It is the latter question we address in this paper.

The main results in the literature pertaining to the boundedness of the Bergman projection on weighted spaces are due to B\'{e}koll\`{e} and Bonami and consider the underlying domain to be the unit ball $\mathbb{B}_n$ \cite{Bek1,Bek2}. The correct condition for the weights, which turns out to be both necessary and sufficient for boundedness on $L^p_{\sigma}(\mathbb{B}_n)$, is referred to as the  B\'{e}koll\`{e}-Bonami, or $B_p$, condition. This weight class is defined using a Muckenhoupt-type condition, but it is slightly altered to reflect the fact that the behavior of the weight away from the boundary is not important. The correct generalization of their condition for an arbitrary domain $D$ is as follows: we say that a weight $\sigma$ belongs to the class $B_p$ if 

$$\sup_{B(w,R); R>d(w,bD)} \left(\frac{1}{V(B(w,R))}\int_{B(w,R)} \sigma \mathop{dV}\right)\left(\frac{1}{V(B(w,R))}\int_{B(w,R)} \sigma^{-1/(p-1)} \mathop{dV}\right)^{p-1}<\infty.$$
Here $V$ refers to Euclidean volume measure. The balls $B(w,R)$ are taken in a quasi-metric that is defined in the interior of the domain $D$ that reflects the boundary geometry.

Considering the Szeg\H{o} projection, there appear to be few weighted results that appear explicitly in the literature. However, from a heuristic point of view, since the Szeg\H{o} projection involves integration on the boundary and is a true singular integral, the correct class of weights should be an adaptation of the $A_p$ Muckenhoupt class in Euclidean harmonic analysis. Therefore, the correct weight condition for the Szeg\H{o} projection to be bounded on $L^p_\sigma(bD)$ should be for $\sigma$ to belong to an $A_p$ class on the boundary, where the non-isotropic boundary ``balls"   reflect the geometry of the domain. In other words, we consider weights $\sigma$ where the following quantity is finite:

$$\sup_{B \subset bD} \left(\frac{1}{\mu(B)}\int_{B}\sigma\mathop{d\mu}\right)\left(\frac{1}{\mu(B)}\int_{B}\sigma^{\frac{-1}{p-1}}\mathop{d\mu}\right)^{p-1}<\infty.$$
The measure $\mu$ in the definition is defined to be Lebesgue surface measure on $bD$. Here $B$ is used to denote a ball in an appropriate quasi-metric that reflects the boundary geometry. In the case that $D=\mathbb{D}$, the ``balls" on the boundary are simply intervals on the circle $\mathbb{T}$, and the $A_p$ condition is the classical one for the boundedness of the Hilbert transform on the circle (see, for example, \cite{Nik}).  We remark that analogous weighted results for the the ball are likely known to the experts. 

\subsection{Statement of Main Results}

The main results of this paper are sufficient conditions on the weights $\sigma$ for the $L^p_\sigma$ boundedness of the Szeg\H{o} and Bergman projections on domains with near-minimal smoothness. This condition on the weights is precisely the $A_p$ condition ($B_p$ condition, respectively) in the setting of spaces of homogeneous type with the appropriate quasi-metric on $bD$ (respectively $D$). We will precisely define these metric quantities in Sections \ref{the problem szego} and \ref{theproblemBergman}. We are able to obtain the result for the Szeg\H{o} projection in the minimal smoothness ($C^2$) case. In the Bergman case, because of a techincal obstruction, we must assume that our domain possesses a $C^4$ boundary. Our two principal results are as follows:

\begin{theorem}\label{szego main2}
Let $D$ be strongly pseudoconvex with $C^2$ boundary. Then for $1<p<\infty$ and $\sigma \in A_p$, the Szeg\H{o} projection $\mathcal{S}$ extends to a bounded operator on $L^p_\sigma(bD)$.

\end{theorem}

\begin{theorem}\label{bergman main}
Let $D$ be strongly pseudoconvex with $C^4$ boundary. Then for $1<p<\infty$ and $\sigma \in B_p$ the Bergman projection $\mathcal{B}$ extends to a bounded operator on $L^p_\sigma(D).$
\end{theorem}

We remark that in the case that $D$ has $C^3$ boundary, our results for the Szeg\H{o} projection can be considerably sharpened. In fact, in this case it is possible to explicitly relate the extension of the Szeg\H{o} projection on the weighted space to the auxiliary operator $\mathcal{C}$ using an operator equation. See Theorem \ref{szego main detailed} in the beginning of Section \ref{szego section} for more details. See also Theorem \ref{bergman main detailed} in the beginning of Section \ref{Bergman section} for a more detailed version of Theorem \ref{bergman main}. 

Note that these theorems only give sufficient conditions, not necessary conditions. Notably, our methods are only suited to proving the sufficiency of the $A_p$/$B_p$ condition, not the necessity. To obtain any results concerning the necessity of the $A_p$/$B_p$ condition, it seems likely one would instead have to study the operator $\mathcal{S}$ or $\mathcal{B}$ directly and obtain novel estimates on the kernel function.

\subsection{An Outline of the Proof} \label{Intro Proof Outline}
For the remainder of the introduction, we provide a broad strokes outline of the method of proof so the reader has an idea of how the various pieces will fit together. Recall that via an idea of Kerzman and Stein, the Szego projection $\mathcal{S}$ can be related to a ``Cauchy integral" $\mathcal{C}$. It can be shown that the operator $\mathcal{C}$ is a (non-orthogonal) projection from $L^2(bD)$ to $H^2(bD)$. Thus, we obtain the following two operator identities relating $\mathcal{S}$ and $\mathcal{C}$ on $L^2(bD)$:
$$\mathcal{S}\mathcal{C}=\mathcal{C}, \quad \mathcal{C} \mathcal{S}=\mathcal{S}.$$

Taking adjoints of the second identity, subtracting from the first and some further manipulation yields the following operator identity:

\begin{equation}\mathcal{S}(I-(\mathcal{C}^*-\mathcal{C}))=\mathcal{C}. \label{1}\end{equation}

We will subsequently refer to \eqref{1} as the \emph{Kerzman-Stein equation}. Note that if $(I-(\mathcal{C}^*-\mathcal{C}))$ is invertible on $L^2(bD)$ (this is true in the case $D$ is $C^\infty$, see \cite{KS2}), we arrive at an explicit formula for $\mathcal{S}$ in terms of $\mathcal{C}$: 
$$\mathcal{S}=\mathcal{C}(I-(\mathcal{C}^*-\mathcal{C}))^{-1}.$$

At this point, it should be noted that a completely analogous approach can be employed for the Bergman projection in which the Cauchy-Fantappi\'e integral, which we denote by $\mathcal{T}$, is taken over the solid domain rather than the boundary (this approach was used to prove certain regularity properties of the Bergman projection; see for example \cite{Lig1,Lig2}). Now perhaps the reader can see the utility of such an approach in proving $L^p$ estimates. To prove that the Szeg\H{o} (or Bergman) projection extends to a bounded operator on $L^p$, one must prove the following two facts concerning $\mathcal{C}$ (respectively $\mathcal{T}$):
\begin{enumerate}
\item The operator $\mathcal{C}$ is bounded on $L^p$;
\item The operator $(I-(\mathcal{C}^*-\mathcal{C}))$ is invertible on $L^p.$
\end{enumerate}

The regularity of the domain is crucial in assessing whether the operator $(I-(\mathcal{C}^*-\mathcal{C}))$ is invertible on $L^p$. If this operator is to be invertible, the ``error" $\mathcal{C}^*-\mathcal{C}$ must be small in some appropriate sense (for example, compact, smoothing, and/or with norm less than $1$). In particular, for the Szeg\H{o} projection, we will require the domain to be $C^3$ (for this method of inversion), while for the Bergman projection we will require the domain to be $C^4$.

As mentioned previously, in \cite{LS1,LS2}, Lanzani and Stein considered the situation of minimal regularity and proved that the Cauchy-Szeg\H{o} and Bergman projections are bounded on $L^p$ for $1<p<\infty$. In the case of the Szeg\H{o} projection they transfer the question of boundedness to real-variable singular integral theory via the theory of spaces of homogeneous type. Recall that a space of homogeneous type is a triple $(X,d, \mu)$ where $X$ is a set, $d$ is a quasi-metric on $X$, and $\mu$ is a measure on $X$ that is doubling on the balls induced by the quasi-metric. Lanzani and Stein show that the kernel of the operator $\mathcal{C}$ satisfies the appropriate size and smoothness estimates with respect to this quasi-metric (more precisely, they consider the kernel of the ``main part" of the operator, $\mathcal{C}^\sharp$; there is an error term they also must handle). The celebrated $T(1)$ Theorem in harmonic analysis is then invoked to establish that the operator $\mathcal{C}$ is bounded on $L^2(bD)$. This result together with the kernel estimates of course implies that $\mathcal{C}$ is bounded on $L^p(bD)$ for $1<p<\infty$. With appropriate control on the ``error term" $\mathcal{C}^*-\mathcal{C}$, Lanzani and Stein establish that $\mathcal{S}$ is bounded on $L^p(bD)$ for $1<p<\infty$. The approach to the Bergman projection is similar insofar as it uses the Kerzman-Stein operator equation, but it is simpler because singular integral theory is not required. Instead, Schur's test for positive operators is a major ingredient in the proof. 

 We follow the general program of Lanzani and Stein in the weighted setting in the next section. In particular, we use the same construction of the auxiliary operator that goes back to Kerzman, Stein, and Ligocka in \cite{KS1,KS2,Lig2}, and we obtain the Kerzman-Stein equation. In the case of the Szeg\H{o} projection, we obtain weighted $L^p$ bounds on the auxiliary operator $\mathcal{C}$ using the same real-variable singular integral approach in \cite{LS1}. The weights belong to an $A_p$ class induced by the quasi-metric on the boundary of $D$.

 In the case of the Bergman projection, Schur's Test is ill-equipped to deal with weights other than radial weights, so a new approach is needed. In particular, to prove the operator $\mathcal{T}$ is boudned on $L^p_\sigma(D)$, we must use a modified singular integral theory and view the Bergman projection as a kind of Calder\'{o}n-Zygmund operator with respect to an appropriate quasi-metric. This idea was precisely the one used by B\'{e}koll\`{e} and Bonami when they obtained weighted $L^p$ estimates for the Bergman projection on the ball when the weight belongs to the $B_p$ class (see \cite{Bek1,Bek2}). Notably, we use key ideas developed by McNeal in \cite{Mc1,Mc2} and other papers that show the Bergman projection can be viewed as a singular integral operator for several important classes of pseudoconvex domains. Combining ideas from these papers, we define a $B_p$ class of weights adapted to our domain and prove that the auxiliary operator $\mathcal{T}$ is bounded on weighted $L^p$. The authors recently took a similar approach when studying the Bergman projection directly in the case when $D$ has smooth boundary, see \cite{HWW,HWW2}. In particular, our results for the Bergman projection in this paper constitute a generalization of the result in \cite{HWW}, because the quasi-metric is the same as the one in that paper. 

In both cases, to show that the operator $(I-(\mathcal{C}^*-\mathcal{C}))$ (or $I-(\mathcal{T}^*-\mathcal{T}))$  is invertible on $L^p_\sigma$ when $\sigma \in A_p$ (or $B_p$), we prove that $\mathcal{C}^*-\mathcal{C}$ (respectively  $\mathcal{T}^*-\mathcal{T}$) is compact on $L^p_\sigma$ for $\sigma\in A_p$ (respectively $B_p$) and also ``improves" $L^p$ spaces. Using the Kerzman-Stein equation, this grants the boundedness of $\mathcal{S}$ (respectively $\mathcal{B}$) on $L^p_\sigma$.

Because Lanzani and Stein assume less regularity, our approach entails an application of their arguments in a simpler setting, so some technical obstructions in their paper can be ignored. In particular, Lanzani and Stein consider an entire family of Cauchy-Fantappi\'{e} type operators $\mathcal{C}_\varepsilon$ with parameter $\varepsilon$, while we only need to consider a single auxiliary operator $\mathcal{C}$ (this can be viewed as a special case of the operators in \cite{LS1} with $\varepsilon=0$). A major technical obstruction in their papers is that the operator $(I-(\mathcal{C}_\varepsilon^*-\mathcal{C}_\varepsilon))$ is no longer invertible on $L^p$, so they must split it appropriately. 

As mentioned previously, we are actually able to obtain the same result for the Szeg\H{o} projection in the case of minimal ($C^2$) smoothness. Here we follow the approach in \cite{LS2} of ``partially inverting" $(I-(\mathcal{C}_\varepsilon^*-\mathcal{C}_\varepsilon))$ by writing $$\mathcal{C}_\varepsilon^*-\mathcal{C}_\varepsilon=\mathcal{A}_\varepsilon+\mathcal{D}_\varepsilon,$$ where $\mathcal{A}_\varepsilon$ has small norm for sufficiently small $\varepsilon$ so $I-\mathcal{A}_\varepsilon$  is invertible on $L^2_\sigma(bD)$ using a Neumann series. We only focus on $p=2$; the general result may be obtained via extrapolation (see \cite{Ru}; extrapolation still holds in spaces of homogeneous type). The result can also be obtained directly without extrapolation, but no new significant information is obtained. The operator $\mathcal{D}_\varepsilon$ may in general be unbounded in norm as $\varepsilon \rightarrow 0$, but it does map $L^2_\sigma(bD)$ to $L^\infty(bD)$, which turns out to be enough. The reverse H\"{o}lder property of $A_p$ weights is the only key property we use in the proof. However, because $B_p$ weights do not satisfy a reverse H\"{o}lder inequality (see \cite{Bori}), we are unable to obtain the same minimal regularity result for the Bergman projection.

This paper is organized as follows. Sections \ref{szego section} and \ref{szego minimal}  are devoted to the Szeg\H{o} projection while Section \ref{Bergman section} focuses on the Bergman projection. Section \ref{szego section} focuses on the case where $D$ is $C^3$ and sharper results can be obtained, while Section \ref{szego minimal} focuses on the minimal smoothness case and proves the full strength of Theorem \ref{szego main2}. At the beginning of each section, the first subsection introduces the background material and the construction of the relevant integral operators. The latter subsections deal with the proofs.

\section{The Szeg\H{o} Projection on $C^3$ domains}\label{szego section}
In this section, we assume that $D$ is a strongly pseudoconvex domain of class $C^3$. We aim to prove the following theorem, which corresponds to a special case of Theorem \ref{szego main2} but also provides more detailed information about the connection between the main and auxiliary operators that is unavailable in the minimal smoothness case.
\begin{theorem}\label{szego main detailed}
Let $D$ be strongly pseudoconvex with $C^3$ boundary. Then for $1<p<\infty$ and $\sigma \in A_p$, the following hold:
\begin{enumerate}
\item The operator $\mathcal{C}^*-\mathcal{C}$ is compact on $L^p_\sigma(bD)$.
\item The operator $I-(\mathcal{C}^*-\mathcal{C})$ is invertible on $L^p_\sigma(bD)$.
\item The Szeg\H{o} projection $\mathcal{S}$ extends to a bounded operator on $L^p_\sigma(bD)$ and satisfies

$$\mathcal{S}=\mathcal{C}(I-(\mathcal{C}^*-\mathcal{C}))^{-1}.$$

\end{enumerate}
\end{theorem}

\subsection{Background and Setup for $C^3$ Domains}\label{the problem szego}

The first step, following the approach of Lanzani and Stein as well as many other authors, is to construct an integral operator that reproduces and produces holomorphic functions from integration of their boundary values. To begin with, define the Levi polynomial at $w \in bD$:
$$P_w(z):= \sum_{j=1}^{n} \dfrac{\partial \rho}{\partial w_j}(w)(z_j-w_j) +\frac{1}{2}\sum_{j,k=1}^{n}\dfrac{\partial^2 \rho}{\partial w_j\partial w_k}(w)(z_j-w_j)(z_k-w_k).$$

Using the strict pseudoconvexity of $D$, it is possible to choose a $C^\infty$ cutoff function $\chi$ and a constant $c$  so that $\chi \equiv 1$ when $|z-w|\leq c/2$ and $\chi \equiv 0$  when $|z-w| \geq c$ so that the function

$$ g(w,z):= \chi(-P_w(z))+(1-\chi)|w-z|^2 $$

\noindent satisfies \begin{equation} \text{Re}(g(w,z)) \gtrsim -\rho(z)+ |w-z|^2 \label{2} \end{equation} for $z \in D$ (see \cite{LS2}). 

Recall that a \emph{generating form} $\eta(w,z)$ is a form of type $(1,0)$ in $w$ with $C^1$ coefficient functions such that $\langle \eta(w,z), w-z \rangle=1$ for all $z \in D$ and $w$ in a neighborhood of $bD$ \cite{LS4}. Here $\langle \cdot, \cdot \rangle$ denotes the action of a $1$-form on a vector in $\mathbb{C}^n$. The importance of generating forms lies in the construction of Cauchy-Fantappi\'e integrals. The upshot of \eqref{2} is that we can construct a generating form as follows: define the following $(1,0)$ form in $w$
$$G(w,z):=\chi \left(\sum_{j=1}^{n} \frac{\partial \rho}{\partial w_j}(w) \mathop{dw_j}-\frac{1}{2}\sum_{j,k=1}^{n} \frac{\partial^2 \rho}{\partial w_j \partial w_k}(w)(w_k-z_k)\mathop{dw_j}\right)+(1-\chi)\sum_{j=1}^{n} (\overline{w}_j-\overline{z}_j)\mathop{dw_j}.$$

Then define for $w \in bD$, $z \in D$ $$\eta(w,z):=\frac{G(w,z)}{\langle G(w,z), w-z \rangle}=\frac{G(w,z)}{g(w,z)}.$$

Then it is immediate that $\eta$ is a generating form. As in \cite{LS1,R}, define the associated Cauchy-Fantappi\'e integral operator

$$\mathbf{C}_1(f)(z):= \frac{1}{(2 \pi \ii)^n} \int_{w \in bD}  f(w)  j^* (\eta \wedge (\overline{\partial}
\eta)^{n-1})(w,z)=\frac{1}{(2 \pi \ii)^n} \int_{w \in bD} \frac{f(w) j^*\left(G \wedge (\overline{\partial}G)^{n-1}(w,z)\right)}{(g(w,z))^{n}},$$
where $j:bD \hookrightarrow \mathbb{C}^n$ is the inclusion map. The point is that this operator reproduces holomorphic functions that are continuous up the boundary, as made precise in the following proposition (see \cite{LS1}):
\begin{proposition}\label{reproduce holo}
Let $F$ be holomorphic on $D$ and continuous on $\overline{D}$, and let $f=F \lvert_{bd}$. Then there holds for $z \in D$
$$\mathbf{C}_1(f)(z)=F(z).$$
\end{proposition}

The problem now is that $\mathbf{C}_1$ does not necessarily produce holomorphic functions, as the form $\eta$ is not necessarily holomorphic in $z$. This difficulty can be overcome by solving a $\overline{\partial}$ problem on a strongly pseudoconvex, smooth domain $\Omega$ that contains $D$ (see \cite{LS1}, or for more details \cite{R}). One has the following:

\begin{proposition} \label{reproduce and produce} There exists an $(n,n-1)$ form (in $w$) $C_2(w,z)$  that is $C^1$ in $w$ and depends smoothly on the parameter $z \in \overline{D}$ so that the following hold for the operator $\mathbf{C}=\mathbf{C}_1+\mathbf{C}_2$:

\begin{enumerate}[label=(\roman*)]
\item $\mathbf{C}(f)(z)=F(z)$ for $F$  holomorphic on $D$ and continuous on $\overline{D}$, where $f=F \lvert_{bd}$;

\item $\mathbf{C}(f)(z)$ is holomorphic for $f \in L^1(bD).$

\end{enumerate}

Here, $$\mathbf{C}_2(f)(z)=\int_{w \in bD}  f(w)  C_2(w,z).$$

\end{proposition}

Note that importantly 
\begin{equation} \sup_{z \in \overline{D}, w \in bD} |C_2(w,z)|<\infty. \label{3} \end{equation} 

Thus, $\mathbf{C}$ is an operator that produces and reproduces holomorphic functions from boundary data. 

\vspace{0.25 cm}

Next, we proceed to define the relevant quasi-metric on the boundary of $D$ for our analysis. Let $d(w,z)=|g(w,z)|^{1/2}$. Then $d(w,z)$ satisfies all the properties of a quasi-metric or quasi-distance. In particular, one has the following, as in \cite{LS1}:
\begin{proposition}
Let $d(w,z)=|g(w,z)|^{1/2}$. Then the following hold for $w,z,\zeta \in bD$:
\begin{enumerate}[label=(\roman*)]
\item $d(w,z) \geq 0$ and $d(w,z)=0$ iff $w=z$;

\item $d(w,z) \approx d(z,w)$;

\item $d(w,z)\lesssim d(w,\zeta)+d(\zeta,z).$
\end{enumerate}
\end{proposition}

By considering the equivalent metric $d(w,z)+d(z,w)$, we might as well assume property (ii) holds with equality (and we make this assumption henceforth). Denote a ball in $bD$ in the quasi-metric with center $z$ and radius $\delta$ by $B(z,\delta)$. It is a fact that 
\begin{equation} \mu(B(z,\delta)) \approx \delta^{2n}.\label{4} \end{equation} where $\mu$ denotes induced Lebesgue surface measure on $bD$.

We also have the important estimates in \cite{LS1}:

\begin{equation} |w-z| \lesssim d(w,z) \lesssim |w-z|^{1/2}. \label{5} \end{equation}

We now introduce the Leray-Levi measure $\lambda$ on $bD$. This measure is defined 

$$\mathop{d\lambda}(w)=j^*(\partial \rho \wedge (\overline{\partial}\partial \rho)^{n-1})/(2 \pi \ii)^n.$$

The use of this measure is crucial in Lanzani and Stein's paper in the computation of an adjoint operator (they do not have apriori boundedness so the existence of the adjoint is not clear), but it turns out to be equivalent to Lebesgue measure in a certain strong sense. In particular, we have \begin{equation}\mathop{d \lambda}(w)=\Lambda(w) \mathop{d \mu}(w) \label{6}, \end{equation} where $\Lambda(w)$ is a function bounded above and below for all $w \in bD$. More explicitly, the function $\Lambda$ is given by

$$\Lambda(w)=(n-1)!(4 \pi)^{-n} |\det{\rho(w)}||\nabla \rho(w)|$$

\noindent where $\det(\rho(w))$ is the determinant of the $(n-1) \times (n-1)$ matrix of second derivatives:

$$\left\{\frac{\partial^2 \rho}{\partial{z_j} \partial{\overline{z}_k}}(w)\right\}_{j,k=1}^{n-1}$$

\noindent and $z=(z_1,z_2,\dots,z_n)$ is computed in a special coordinate system (see \cite{LS1} for details). Crucially, note that $\Lambda$ is Lipschitz, since $\rho$ is of class $C^3$. The importance of this fact will become clear in the proof of Lemma \ref{kernel difference}.

 An important result, also in \cite{LS1}, is as follows:
\begin{proposition}
The triple $(bD, d, \mathop{d \lambda})$ forms a space of homogeneous type (in the sense of the theory of singular integrals).
\end{proposition}

Note we could replace the Leray-Levi measure by induced Lebesgue measure and the above result would still be true, since the function $\Lambda(w)$ is bounded above and below uniformly. Below, for a measurable set $S \subset bD$, when we write $\mu(S)$, we refer to its Lebesgue surface measure, but in every case we could replace it by the Leray-Levi measure and the result would still be true. 

\vspace{0.25 cm}

We now want to essentially consider the restriction of the operator $\mathbf{C}$ to the boundary $bD$ and obtain a singular integral operator $\mathcal{C}$ that maps $L^p(bD)$ to $L^p(bD)$. Explicitly, Lanzani and Stein define

$$\mathcal{C}(f)(z)=\mathbf{C}(f)(z)\lvert_{bD}$$

\noindent when $f$ satisfies a type of H\"{o}lder continuity, namely 

$$|f(w_1)-f(w_2)|\lesssim d(w_1,w_2)^\alpha$$ for some $\alpha$ with $0<\alpha \leq 1$. In this case one can show $\mathbf{C}(f)$ extends to a continuous function on $\overline{D}$, so the above definition makes sense. The operator $\mathcal{C}$, while initially defined only on certain functions, actually extends to a bounded linear operator on $L^p(bD)$ (this is proven in \cite{LS1} using the $T(1)$ theorem). 

\vspace{0.25 cm}

Now, it is useful to break the operator $\mathbf{C}$ into a main term and an error term as follows: $$\mathbf{C}=\mathbf{C}^\sharp+\mathbf{R},$$ where $$\mathbf{C}^\sharp(f)(z)=\int_{bD} \frac{f(w)}{g(w,z)^n}\mathop{d \lambda}(w)$$ and $\mathbf{R}$ absorbs the error from replacing the numerator of the Cauchy-Fantappi\'e integral with the Leray-Levi measure as well as the error from the operator $\mathbf{C}_2$, which in fact has a bounded kernel by (\ref{2}). If we let $R(w,z)$ denote the kernel of the operator $\mathbf{R}$, we can obtain the crucial estimate (see \cite{LS1} again):
\begin{equation} |R(w,z)| \lesssim d(w,z)^{-2n+1}. \label{7} \end{equation}

Note that it is immediately obvious that the kernel of $\mathbf{C}^\sharp$ is bounded above by a multiple of $d(w,z)^{-2n}$, so we see that the operator $\mathbf{R}$ is ``less singular" in a sense than the operator $\mathbf{C}^\sharp$. 

\vspace{0.25 cm}

As before, for functions that satisfy the H\"{o}lder continuity condition as above, we can define
$$\mathcal{C}^\sharp(f)=\mathbf{C}^\sharp(f)\lvert_{bD}$$

\noindent and thus obtain the decomposition for the operator $\mathcal{C}$

$$\mathcal{C}=\mathcal{C}^\sharp+\mathcal{R}.$$

Finally, we define the so-called $A_p$ classes of weights on $bD$ for $1<p<\infty$ with respect to the quasi-metric:
\begin{definition}  A function $\sigma \in L^1(bD)$ that is positive almost everywhere is said to belong to the class $A_p$ if the following quantity is finite:

$$[\sigma]_p:= \sup_{B \subset bD} \left(\frac{1}{\mu(B)}\int_{B}\sigma\mathop{d\mu}\right)\left(\frac{1}{\mu(B)}\int_{B}\sigma^{\frac{-1}{p-1}}\mathop{d\mu}\right)^{p-1}$$

\noindent where $B$ is a ball in the quasi-metric $d$. 
\end{definition}

Additionally, we can define a suitable maximal function with respect to this quasi-metric on $bD$:

\begin{definition}
The \emph{Hardy-Littlewood Maximal Function} is defined, for $f \in L^1(bD)$
$$\mathcal{M}(f)(z)=\sup_{B\ni z}\frac{1}{\mu(B)}\int_{B}|f(w)|\mathop{d \mu(w)}$$

\noindent where as before $B$ is a ball in the quasi-metric $d$. 
\end{definition}

We also define $A_1$ weights with respect to the same quasi-metric:

\begin{definition}  A function $\sigma \in L^1(bD)$ that is positive almost everywhere is said to belong to the class $A_1$ if the following estimate holds for all $z \in bD$:
$$\mathcal{M}(\sigma)(z) \lesssim \sigma(z).$$
\end{definition}

We have now set up all the machinery we need to prove Theorem \ref{szego main detailed}.
\subsection{The Main Term}
We proceed to analyze the ``main term" $\mathcal{C}^\sharp$. It should be noted in what follows that in the $C^2$ case considered in \cite{LS1}, certain implicit constants depend on $\varepsilon$ and can even blow up as $\varepsilon \rightarrow 0$. This is not the case in the $C^3$ case, as there is only one $\varepsilon$, namely $\varepsilon=0$, for which there is no analog in the $C^2$ case. We have the following size and smoothness estimates for the kernel of  $\mathcal{C}^\sharp$ given in \cite{LS1}:
\begin{proposition}\label{kernel estimates}
Let $K(z,w)=g(w,z)^{-n}$ denote the kernel of $\mathcal{C}^\sharp$ with respect to the Leray-Levi measure. Then there holds:
\begin{enumerate}[label=(\roman*)]

\item $|K(z,w)|\lesssim d(w,z)^{-2n};$

\item $|K(z,w)-K(z,w')| \lesssim \frac{d(w,w')}{d(w,z)^{2n+1}}$ for $d(w,z) \geq c d(w,w');$

\item $|K(z,w)-K(z',w)| \lesssim \frac{d(z,z')}{d(w,z)^{2n+1}}$ for $d(w,z) \geq c d(z,z'),$

\end{enumerate}

\noindent where $c$ is some appropriately large constant.
\end{proposition}

Lanzani and Stein also prove the following result by invoking the $T(1)$ theorem:
\begin{theorem}\label{T(1)}
The operator $\mathcal{C}^\sharp$ is bounded on $L^2(bD)$.
\end{theorem}

Theorem \ref{T(1)} and Proposition \ref{kernel estimates} demonstrate that the operator $\mathcal{C}^\sharp$ is Calder\'{o}n-Zygmund in the sense of spaces of homogeneous type, and consequently the weighted theory of real-variable harmonic analysis applies to this case. Thus, we have the following result:

\begin{theorem}\label{c sharp bound}
Let $\sigma \in A_p$, where $A_p$ is defined as above. The operator $\mathcal{C}^\sharp$ is bounded from $L^p_\sigma(bD)$ to $L^p_\sigma(bD)$, $1<p<\infty$.
\begin{proof}
This is an easy consequence of classical singular integral theory on spaces of homogeneous type. The only remark that needs to be made is that the equivalence of the Leray-Levi measure and Lebesgue measure in (\ref{6}) must be invoked because the kernel above is with respect to Leray-Levi measure, not Lebesgue measure. In particular, if $\sigma \in A_p$ as we have defined it, then $\sigma$ is in $A_p$ with respect to the Leray-Levi measure. By Calder\'{o}n-Zygmund theory on spaces of homogeneous type, the operator $\mathcal{C}^\sharp$ is bounded on $L^p(bD,\sigma \mathop{d \lambda})$, and hence bounded on $L^p(bD, \sigma \mathop{d \mu})$ by the equivalence of the measures.

\end{proof}

\end{theorem}

\subsection{The Error Terms}
Let $\mathcal{C}^*$ denote the adjoint of $\mathcal{C}$ with respect to Lebesgue measure. We now proceed to deal with the error terms $\mathcal{R}$ as well as $\mathcal{C}^*-\mathcal{C}$. Both of these terms will play a role in the proof of the main theorem in the subsequent section. We know from \eqref{7} that the kernel of the ``remainder operator" $\mathcal{R}$ is ``less singular" than the main operator $\mathcal{C}^\sharp$. We proceed to show that this is also true for the kernel of the ``difference operator" $\mathcal{C}^*-\mathcal{C}.$ First we need a preliminary proposition, which is similar to an argument that can be found in \cite{R}:

\begin{proposition}\label{diff of Levi}
The following estimate holds for $w,z \in bD$:
$$|g(w,z)-\overline{g(z,w)}| \lesssim |w-z|^3.$$
\begin{proof}
It suffices to prove the estimate when $|w-z|\leq c/2$, so we can assume $g(w,z)=-P_w(z)$ and $\overline{g(z,w)}=-\overline{P_z(w)}$. To avoid cumbersome notation, we use the shorthand $\frac{\partial \rho}{\partial w_j}(w)=\rho_j(w)$ and $ \frac{\partial^2 \rho}{\partial w_j \partial w_k}(w)=\rho_{j,k}(w).$ Recall the Levi polynomial at $w$ is defined as

$$P_w(z)= \sum_{j=1}^{n} \rho_j(w)(z_j-w_j) +\frac{1}{2}\sum_{j,k=1}^{n}\rho_{j,k}(w)(z_j-w_j)(z_k-w_k).$$ We also define the Levi form

$$L_w(z)=\sum_{j,k=1}^{n}\dfrac{\partial^2 \rho}{\partial w_j \partial \overline{w}_k}(w)(z_j-w_j)(\overline{z}_k-\overline{w}_k).$$

 The Taylor expansion (in $w$) of $\rho_j(w)$ about $w=z$ is 
$$\rho_j(w)=\rho_j(z)+ \sum_{k=1}^n\rho_{j,k}(z)(w_k-z_k) + \sum_{k=1}^n\frac{\partial^2 \rho}{\partial z_j \partial \overline{z_k}}(z)(\overline{w_k}-\overline{z_k})+\mathcal{O}(|w-z|^2)$$ while the Taylor expansion of $\rho_{j,k}(w)$ gives $$\rho_{j,k}(w)=\rho_{j,k}(z)+\mathcal{O}(|w-z|).$$

 Substituting these Taylor expansions into $P_w(z)$, we obtain
 \begin{small}
 $$P_w(z)=\sum_{j=1}^{n} \rho_j(z)(z_j-w_j)-\frac{1}{2}\sum_{j,k=1}^n \rho_{j,k}(z)(w_j-z_j)(w_k-z_k)-
 \sum_{j,k=1}^{n}\frac{\partial^2 \rho}{\partial z_j \partial \overline{z_k}}(z)(w_j-z_j)(\overline{w_k}-\overline{z_k})+\mathcal{O}(|w-z|^3).$$
 \end{small}
 
 On the other hand, we have

$$\overline{P_z(w)}= \sum_{j=1}^{n} \overline{\rho_j(z)}(\overline{w_j}-\overline{z_j}) +\frac{1}{2}\sum_{j,k=1}^{n}\overline{\rho_{j,k}(z)}(\overline{w_j}-\overline{z_j})(\overline{w_k}-\overline{z_k}).$$ 

\noindent A computation shows
$$\overline{P_z(w)}-P_w(z)=2 \text{Re}P_z(w)+L_z(w)+\mathcal{O}(|w-z|^3).$$

\noindent Then just use the well-known fact that 
$$\rho(w)=\rho(z)+2 \text{Re}P_z(w)+L_z(w)+\mathcal{O}(|w-z|^3),$$ together
 with the fact that $\rho(z)=\rho(w)=0$ as $w,z \in bD$.

\end{proof}
\end{proposition}

This proposition will allow us to prove the following lemma. Again, the argument is essentially from \cite{R}.

\begin{lemma}\label{kernel difference}
Let $K(z,w)$ denote the kernel of $(\mathcal{C}^\sharp)^*-\mathcal{C}^\sharp$ with respect to Lebesgue measure $\mathop{d \mu}$. Then the following estimate holds:
$$|K(z,w)| \lesssim d(w,z)^{-2n+1}.$$
\begin{proof} 
Here we need to come to grips with the distinction between the Leray-Levi measure $\mathop{d \lambda}$ and the Lebesgue measure $\mathop{d\mu}$. Note that if $(\mathcal{C}^\sharp)^\dagger$ denotes the adjoint of $\mathcal{C}^\sharp$ taken \emph{with respect to the Leray-Levi measure}, then we have the relation $(\mathcal{C}^\sharp)^\dagger=\Lambda  (\mathcal{C}^\sharp)^* \Lambda^{-1}$ (see \cite{LS1}). Let $K_L(w,z)$ denote the kernel, with respect to $\mathop{d\lambda}$, of the operator $(\mathcal{C} ^\sharp)^{\dagger}-\mathcal{C}^\sharp$. It is immediate that $K_L(w,z)=\overline{g(z,w)}^{-n}-g(w,z)^{-n}$. Compute to see
\begin{eqnarray*}
|K(z,w)|& = & \left|\Lambda(z)[\overline{g(z,w)}^{-n}\Lambda(w)]\Lambda^{-1}(w) )-g(w,z)^{-n}\Lambda(w)\right| \\
& = & \left|\Lambda(z) \overline{g(z,w)}^{-n}-g(w,z)^{-n}\Lambda(w)\right|\\
& \leq & |\Lambda(z)-\Lambda(w)| |\overline{g(z,w)}|^{-n}+ |\Lambda(w)||K_L(w,z)|\\
& \lesssim & |z-w| d(w,z)^{-2n} +|K_L(w,z)|\\
&  \lesssim & d(w,z)^{-2n+1} +|K_L(w,z)| . 
\end{eqnarray*}

\noindent Here we use the fact that $\Lambda$ is Lipschitz. Then, compute to see:
\begin{eqnarray*}
|K_L(z,w)|& = &  \left|\overline{g(z,w)}^{-n}-g(w,z)^{-n}\right| \\
& = & \left|\frac{g(w,z)^n-\overline{g(z,w)}^{n}}{g(w,z)^n\overline{g(z,w)}^{n}}\right|\\
& = & \left|\frac{\left(g(w,z)-\overline{g(z,w)}\right)\left(\sum_{t=0}^{n-1} (g(w,z))^t(\overline{g(z,w)})^{n-1-t}\right)}{g(w,z)^n\overline{g(z,w)}^{n}}\right|\\
& \lesssim & \frac{|g(w,z)-\overline{g(z,w)}|d(w,z)^{2n-2}}{d(w,z)^{4n}}\\
& \lesssim & d(w,z)^{-2n+1}
\end{eqnarray*}

\noindent where in the last estimation we used Proposition \ref{diff of Levi}.  
\end{proof}
\end{lemma}

One can show using a special coordinate system that
\begin{equation}\sup_{z \in bD} \int_{bD}d(w,z)^{-2n+1} \mathop{d \mu(w)} <\infty \label{8} \end{equation}
 (see \cite{LS1} or \cite{R}). This result can also be obtained by integrating over dyadic ``annuli" as we will later see. Thus, we see that $\mathcal{R}$ and $(\mathcal{C}^\sharp)^*-\mathcal{C}^\sharp$ have integrable kernels, while $\mathcal{C}^\sharp$ does not. 
 
 \vspace{0.25 cm}

Now we show the these kernel estimates are not only enough to guarantee boundedness on weighted $L^p$ spaces; they are actually enough to guarantee compactness which is much better. The following proposition allows for good control of the integration of an $A_1$ weight $\sigma$ against a kernel $K(z,w)$ which satisfies the size estimate above. 

\begin{proposition}\label{A_1 control}
Let $K(z,w)$ be a kernel measurable on $bD \times bD$ that satisfies the size estimate $|K(z,w)| \lesssim d(w,z)^{-2n+1}$, and let $\sigma \in A_1$. Then the following estimates hold for all $z,w \in bD$:

$$  \int_{B(z,\delta)}|K(z,w)| \sigma(w) \mathop{d \mu(w)}\lesssim \delta \sigma(z)$$

$$  \int_{B(w,\delta)}|K(z,w)| \sigma(z) \mathop{d \mu(z)} \lesssim \delta \sigma(w).$$

\begin{proof}
Break the region of integration up into dyadic annuli and estimate the integral as follows:
\begin{align*}
& \int_{B(z,\delta)}|K(z,w)| \sigma(w) \mathop{d \mu(w)} \\
& \lesssim  \int_{B(z,\delta)}d(w,z)^{-2n+1} \sigma(w) \mathop{d \mu(w)}\\
& = \sum_{i=0}^{\infty} \int_{B(z,2^{-i}\delta)\setminus B(z,2^{-(i+1)}\delta)}d(w,z)^{-2n+1} \sigma(w) \mathop{d \mu(w)}\\
& \leq  \sum_{i=0}^{\infty} \int_{B(z,2^{-i}\delta)\setminus B(z,2^{-(i+1)}\delta)}2^{(-(i+1)(-2n+1))}\delta^{(-2n+1)} \sigma(w) \mathop{d \mu(w)}\\
& \leq \sum_{i=0}^{\infty} 2^{(-(i+1)(-2n+1))}\delta^{(-2n+1)}\mu(B(z,2^{-i}\delta)) \frac{1}{\mu(B(z,2^{-i}\delta))}\int_{B(z,2^{-i}\delta)} \sigma(w) \mathop{d \mu(w)}\\
& \leq  \sum_{i=0}^{\infty} 2^{2n-1} 2^{-i} \delta \mathcal{M}({\sigma})(z)\\
& \lesssim  \delta \mathcal{M}(\sigma)(z)\\
& \lesssim  \delta \sigma(z).
\end{align*}

Note all implicit equivalences are independent of $w$ and $z$. The proof of the other statement is completely analogous.

\end{proof}

\end{proposition}

Note if $K(z,w)$ is the kernel of an integral operator satisfying the size estimate of the previous proposition, then $K$ is ``integrable" in the sense that

\begin{equation}\sup_{z \in bD} \int_{bD}|K(z,w)| \mathop{d \mu(w)} <\infty,\label{9}\end{equation}

\noindent and obviously \eqref{9} still holds if the roles of $z$ and $w$ are interchanged. This can be seen by taking $\sigma=1$ and $\delta$ sufficiently large in Proposition \ref{A_1 control}. But in fact, we can say something slightly better. The proof of the following proposition is essentially a reprise of Proposition \ref{A_1 control} taking $\sigma=1$ with obvious modifications.

\begin{proposition}\label{better kernel}
Let $K(z,w)$ be a kernel measurable on $bD \times bD$ that satisfies the size estimate $|K(z,w)| \lesssim d(w,z)^{-2n+1}$, and let $\varepsilon \in [0, \frac{1}{2n-1})$. Then the following hold:
$$\sup_{z \in bD} \int_{bD}|K(z,w)|^{1+\varepsilon} \mathop{d \mu(w)} <\infty$$

$$\sup_{w \in bD} \int_{bD}|K(z,w)|^{1+\varepsilon} \mathop{d \mu(z)} <\infty.$$

\end{proposition}

As a consequence of this proposition, we can prove that an integral operator $\mathcal{K}$ that has a kernel with the above size estimate ``improves" $L^p$ spaces. This was noted before in \cite{KS2} using a slightly different approach. 

\begin{proposition}\label{Lp improvement}
Let $\mathcal{K}$ be an integral operator on $L^p(bD)$ with a kernel $K(z,w)$ that satisfies the size estimate $|K(z,w)| \lesssim d(w,z)^{-2n+1}.$ Then $\mathcal{K}$ maps $L^p(bD)$ to $L^{p+\varepsilon}(bD)$ boundedly for $p \geq 1$ and $\varepsilon \in [0, \frac{1}{2n-1}).$

\begin{proof}

We first demonstrate the result for $p=1$ and then show how this implies the result for $p>1$. Take $f \in L^1(bD)$ and $\varepsilon \in [0, \frac{1}{2n-1}).$ Then compute, using Minkowski's integral inequality and Proposition \ref{better kernel}:

\begin{eqnarray*}
\left( \int_{bD} \left| \int_{bD} K(z,w) f(w) \mathop{d \mu(w)} \right|^{1+\varepsilon} \mathop{d \mu(z)} \right)^{\frac{1}{1+ \varepsilon}}& \leq & \left(\int_{bD} \left( \int_{bD} |K(z,w)| |f(w)| \mathop{d \mu(w)} \right)^{1+\varepsilon} \mathop{d \mu(z)}\right)^{\frac{1}{1+\varepsilon}}\\
& \leq & \int_{bD} \left( \int_{bD} |K(z,w)|^{1+\varepsilon} \mathop{d \mu(z)} \right)^{\frac{1}{1+\varepsilon}} |f(w)| \mathop{d \mu(w)}\\
& \lesssim & ||f||_{L^1(bD)}.
\end{eqnarray*}

To obtain the result for $p>1$, proceed as follows, using H\"{o}lder's inequality with exponents $p$ and $q$:

\begin{align*}
& \left(\int_{bD} \left| \int_{bD} K(z,w) f(w) \mathop{d \mu(w)} \right|^{p+\varepsilon} \mathop{d \mu(z)}\right)^{\frac{1}{p+\varepsilon}}   \\
& \leq \left(\int_{bD} \left( \int_{bD} |K(z,w)|^{1/p}|K(z,w)|^{1/q} |f(w)| \mathop{d \mu(w)} \right)^{p+\varepsilon} \mathop{d \mu(z)}\right)^{\frac{1}{p+\varepsilon}}\\
 & \leq \left(\int_{bD} \left(\int_{bD}|K(z,w)| \mathop{d \mu(w)}\right)^{\frac{p+\varepsilon}{q}} \left(\int_{bD}|K(z,w)| |f(w)|^p \mathop{d \mu(w)}\right)^{\frac{p+\varepsilon}{p}} \mathop{d \mu(z)}\right)^{\frac{1}{p+\varepsilon}} \\
 & \lesssim  \left(\int_{bD} \left(\int_{bD}|K(z,w)| |f(w)|^p \mathop{d \mu(w)}\right)^{\frac{p+\varepsilon}{p}} \mathop{d \mu(z)}\right)^{\frac{1}{p+\varepsilon}}\\
&  =  \left(\left(\int_{bD} \left(\int_{bD}|K(z,w)| |f(w)|^p \mathop{d \mu(w)}\right)^{\frac{p+\varepsilon}{p}} \mathop{d \mu(z)}\right)^{\frac{p}{p+\varepsilon}}\right)^{\frac{1}{p}}\\
& \leq  \left(\int_{bD} \left(\int_{bD}|K(z,w)|^{1+\frac{\varepsilon}{p} } \mathop{d \mu(z)}\right)^{\frac{p}{p+\varepsilon}} |f(w)|^p\mathop{d \mu(w)}\right)^{\frac{1}{p}}\\
& \lesssim  ||f||_{L^p(bD)}.
\end{align*}

\noindent In the penultimate line, notice we apply Minkowski's integral inequality with exponent $\frac{p+\varepsilon}{p}=1+\frac{\varepsilon}{p}$ and with respect to measures $|f(w)|^p \mathop{d \mu(w)}$ and $\mathop{d \mu(z)}$.

\end{proof}

\end{proposition}

\noindent Thus, we obtain the following important corollary:

\begin{corollary}\label{improving operators}
 The operators $\mathcal{R}$, $\mathcal{R}^*$, and $(\mathcal{C}^\sharp)^*-\mathcal{C}^\sharp$   map $L^p(bD)$ to $L^{p+\varepsilon}(bD)$ for $p \geq 1$ and $\varepsilon \in [0, \frac{1}{2n-1}).$
\end{corollary}

We now turn to a proof of the major lemma concerning the error terms. This lemma adapts an argument that can be found in \cite{R} to the weighted setting. It also should be noted that components of this proof are analogous to a ``weighted Schur test," which appears to be well known, see, for example, \cite{Tao}.

\begin{lemma}\label{compactness} Let $K(z,w)$ be a measurable function on $bD \times bD$ satisfying the following for all $z,w \in bD$ and all weights $\sigma \in A_1$:
\begin{enumerate}[label=(\roman*)]
\item $ \int_{B(z,\delta)}|K(z,w)| \sigma(w) \mathop{d \mu(w)} \lesssim C(\delta) \sigma(z);$

\item $ \int_{B(w,\delta)}|K(z,w)| \sigma(z) \mathop{d \mu(z)} \lesssim C(\delta) \sigma(w);$

\item For any fixed $\delta>0$, the kernel $K(z,w)$ is bounded on $$bD \times bD \setminus \{(z,w): d(z,w)<\delta\}$$ (with a bound that depends on $\delta$).
\end{enumerate}

\noindent Furthermore, $C(\delta)$ tends to $0$ as $\delta \rightarrow 0$. Then the operator $\mathcal{K}$ defined by $\mathcal{K}(f)(z)=\int_{bD} K(z,w) f(w) \mathop{d \mu(w)}$ is compact on $L^p_\sigma(bD)$ for $\sigma \in A_p$.
\begin{proof}

First, consider the case when $K$ is bounded on $bD \times bD$, say $||K||_{L^\infty(bD)} \leq M$. Let $\sigma \in A_p$. Then note that the kernel of the operator $\mathcal{K}$ with respect to the weighted measure $\mathop{d \sigma}= \sigma \mathop{d \mu}$ is $\tilde{K}(z,w)=K(z,w)\sigma^{-1}(w)$. To prove compactness on $L^p_\sigma(bD)$, it suffices to show the following double-norm is finite (it is well-known the finiteness of this double-norm implies compactness, for example see \cite{Eveson}):

$$\int_{bD} \left(\int_{bD} |\tilde{K}(z,w)|^q \mathop{d \sigma(w)}\right)^{p/q} \mathop{d \sigma(z)}, $$ where $q$ denotes the H\"{o}lder exponent conjugate to $p$.
 Then we have

\begin{eqnarray*} 
\int_{bD} \left(\int_{bD} |\tilde{K}(z,w)|^q \mathop{d \sigma(w)}\right)^{p/q} \mathop{d \sigma(z)} & = & \int_{bD} \left(\int_{bD} |K(z,w)|^q \sigma^{-\frac{1}{p-1}} \mathop{d \mu(w)}\right)^{p/q} \sigma(z) \mathop{d \mu(z)} \\
& \leq & M^p ||\sigma||_{L^1(bD)} ||\sigma^{\frac{-1}{p-1}}||_{L^1(bD)}^{p-1}\\
& < & \infty
\end{eqnarray*} since 
$\sigma,\sigma^{-\frac{1}{p-1}}$ are integrable on $bD$. Thus the theorem holds in this case.

To pass to the case where $K$ is unbounded, let $\delta_j=\frac{1}{j}$ and
$$K_j(z,w)=
\begin{cases}
K(z,w) & d(w,z) \geq \delta_j \\
0 & d(w,z)<\delta_j
\end{cases}.$$

Let $\mathcal{K}_j$ be the integral operator with kernel $K_j$. Then, by hypothesis $K_j$ is bounded on $bD \times bD$ and by the argument above, $\mathcal{K}_j$ is compact on $L^p_\sigma(bD)$. Since the compact operators are a closed subspace of the Banach space of bounded linear operators on $L^p_\sigma(bD)$, if we can show that the operators $\mathcal{K}_j$ approach $\mathcal{K}$ in operator norm, we will be done. 

\vspace{0.25 cm}

To this end, let $f \in L^p_\sigma(bD)$ with $||f||_{L^p_\sigma(bD)}\leq 1.$ Note that as $\sigma \in A_p$, we can write $$\sigma=\frac{\sigma_1}{\sigma_2^{p-1}}$$ where $\sigma_1,\sigma_2 \in A_1$ by the factorization of $A_p$ weights in the setting of spaces of homogeneous type (see, for example, \cite{Ru} for a proof of this well-known fact). By H\"{o}lder's Inequality applied to the functions $|K(z,w)-K_j(z,w)|^{1/q}\sigma_2(w)^{1/q}$ and $|K(z,w)-K_j(z,w)|^{1/p}\sigma_2(w)^{-1/q}|f(w)|$ and then applying Proposition \ref{A_1 control}, we obtain the estimate:
\begin{small}
\begin{eqnarray*}
|(\mathcal{K}-\mathcal{K}_j)(f)(z)|& \leq & \int_{bD}|K(z,w)-K_j(z,w)||f(w)| \mathop{d \mu(w)}\\
& =& \left(\int_{B(z, \delta_j)} |K(z,w)| \sigma_2(w) \mathop{d \mu(w)}\right)^{1/q} \left(\int_{B(z,\delta_j)} |K(z,w)| (\sigma_2(w))^{1-p}|f(w)|^p \mathop{d \mu(w)} \right)^{1/p}\\
& \lesssim & C(\delta_j)^{1/q} \sigma_2(z)^{1/q} \left(\int_{B(z,\delta_j)} |K(z,w)| (\sigma_2(w))^{1-p}|f(w)|^p \mathop{d \mu(w)} \right)^{1/p}.\\
\end{eqnarray*}
\end{small}

Thus, we obtain, applying the proceeding estimate, Fubini, and Proposition \ref{A_1 control} again: 
\begin{eqnarray*}
||(\mathcal{K}-\mathcal{K}_j)f||^p_{L^p_\sigma(bD)} & \leq &\int_{bD} C(\delta_j)^{\frac{p}{q}} \sigma_2(z)^{\frac{p}{q}} \left(\int_{B(z,\delta_j)} |K(z,w)| (\sigma_2(w))^{1-p} |f(w)|^p \mathop{d \mu(w)} \right)\frac{\sigma_1(z)}{\sigma_2(z)^{p-1}} \mathop{d \mu(z)}\\
& = & C(\delta_j)^{\frac{p}{q}} \int_{bD} \int_{B(z,\delta_j)} |K(z,w)| (\sigma_2(w))^{1-p} |f(w)|^p \mathop{d \mu(w)} \sigma_1(z)\mathop{d \mu(z)}\\
& = & C(\delta_j)^{\frac{p}{q}} \int_{bD} \left(\int_{B(w,\delta_j)} |K(z,w)| \sigma_1(z)\mathop{d \mu(z)}\right) |f(w)|^p  (\sigma_2(w))^{1-p}\mathop{d \mu(w)}\\
& \lesssim & C(\delta_j)^p \int_{bD} \sigma_1(w) |f(w)|^p  (\sigma_2(w))^{1-p}\mathop{d \mu(w)}\\
& = & C(\delta_j)^p ||f||^p_{L^p_\sigma(bD)}\\
& \leq & C(\delta_j)^p.
\end{eqnarray*}

Letting $j \rightarrow \infty$, we have $\delta_j \rightarrow 0$ and $C(\delta_j) \rightarrow 0$. Thus, it immediately follows that the operators $\mathcal{K}_j$ approach $\mathcal{K}$ in operator norm and hence $\mathcal{K}$ is compact.

\end{proof}
\end{lemma}

The preceding lemma admits the following, very useful corollary:
 \begin{corollary}\label{compact operators}
 The operators $\mathcal{R}$, $\mathcal{R}^*$, and $(\mathcal{C}^\sharp)^*-\mathcal{C}^\sharp$ are compact on $L^p_\sigma(bD)$ for $\sigma \in A_p$. 
 \end{corollary}

We need one more crucial lemma to conclude our analysis of the error terms and allow us to present the proof of Theorem \ref{szego main detailed} in the next section.

\begin{lemma}\label{spectrum}
Let $\mathcal{K}$ be an integral operator on $L^p(bD)$ with a kernel $K(z,w)$ that satisfies the size estimate $|K(z,w)| \lesssim d(w,z)^{-2n+1}.$ Further suppose that $\ii\mathcal{K}$ is self-adjoint on $L^2(bD)$.  Then $1$ is not in the spectrum of $\mathcal{K}$ considered as an operator on $L^p_\sigma(bD)$, where $\sigma$ is an $A_p$ weight. 

\begin{proof}
First, note that $1$ is not an eigenvalue of $\mathcal{K}$ considered as an operator on (unweighted) $L^2(bD)$. So suppose to the contrary that there exists an eigenfunction $f \in L^p_\sigma(bD)$ such that $\mathcal{K}f=f$. We assert $f \in L^1(bD)$. To see this, note by H\"{o}lder
\begin{eqnarray*}
\int_{bD}|f(w)| \mathop{d \mu(w)}& = & \int_{bD}|f(w)| \sigma(w)^{1/p}\sigma(w)^{-1/p}\mathop{d \mu(w)}\\
& \leq & ||f||_{L^p_{\sigma}(bD)}||\sigma^{-\frac{1}{p-1}}||^{1/q}_{L^{1}(bD)}\\
& < & \infty.
\end{eqnarray*}

Then, by Corollary \ref{improving operators}, $f \in L^{1+\varepsilon}(bD)$. In particular, we have
\begin{eqnarray*}
||f||_{L^{1+\varepsilon}(bD)} & = & ||\mathcal{K}f||_{L^{1+\varepsilon}(bD)}\\
& \lesssim & ||f||_{L^1(bD)}\\
& < & \infty.
\end{eqnarray*}

 But since $\mathcal{K}f=f$, we can repeat this argument to obtain $f \in L^{1+2\varepsilon}$. In fact, we can iterate this argument arbitrarily many times to obtain that $f \in L^p(bD)$ for all $p \geq 1$! In particular, $f \in L^2(bD)$. This contradicts the fact that $1$ is not an eigenvalue of $\mathcal{K}$ on $L^2(bD).$ Since $\mathcal{K}$ is compact on $L^p_{\sigma}(bD)$ by Corollary \ref{compact operators} (or rather the arguments leading to this corollary), this implies $1$ is not in the spectrum of $\mathcal{K}$ on $L^p_\sigma(bD)$, as required.

\end{proof}

\end{lemma}

\subsection{Proof of Theorem \ref{szego main detailed}}

Equipped with these definitions and results, we are in a position to prove Theorem \ref{szego main detailed}. As discussed, the essential ideas in this proof have been around for a long time and can be found, for example, in \cite{KS1,KS2}.

\begin{proof}[Proof of Theorem~\ref{szego main detailed}] 
First, note that both $\mathcal{S}$ and $\mathcal{C}$ essentially produce and reproduce boundary values of holomorphic functions: they are projections onto $H^2(bD)$ (this is proven precisely in \cite{LS3}). Consequently, we obtain the following two operator identities on $L^2(bD)$: $\mathcal{S}\mathcal{C}=\mathcal{C}$ and $\mathcal{C}\mathcal{S}=\mathcal{S}$. Taking adjoints of the second identity and using the fact that the Szeg\H{o} projection is self-adjoint, we get $\mathcal{S}\mathcal{C}^*=\mathcal{S}$, and further manipulation yields $\mathcal{S}(\mathcal{C}^*-\mathcal{C})=\mathcal{S}-\mathcal{C},$ or $\mathcal{S}(I-\mathcal{A})=\mathcal{C}$ where $I$ denotes the identity operator and $\mathcal{A}=\mathcal{C}^*-\mathcal{C}$. By Theorem \ref{c sharp bound} and Corollary \ref{compact operators} and, we know that $\mathcal{C}=\mathcal{C}^\sharp+\mathcal{R}$ is bounded on $L^p_\sigma(bD)$ for $\sigma\in A_p$. 

\vspace{0.2 cm}

Next, we assert that the operator $\mathcal{A}$ is compact on $L^p_\sigma(bD)$. To see this, write $$\mathcal{A}=(\mathcal{C}^\sharp)^*-\mathcal{C}^\sharp+(\mathcal{C}^\sharp-\mathcal{C})+(\mathcal{C}^*-(\mathcal{C}^\sharp)^*)=((\mathcal{C}^\sharp)^*-\mathcal{C}^\sharp)-\mathcal{R}+\mathcal{R}^*$$ and appeal to Corollary \ref{compact operators}. Next, an easy computation shows that $\ii \mathcal{A}$ is self-adjoint on $L^2(bD)$. It follows from Lemma \ref{spectrum} that $1$ is not in the spectrum of $\mathcal{A}$ considered as an operator on $L^p_\sigma(bD)$ and hence the operator $(I-\mathcal{A})$ is invertible on $L^p_\sigma(bD)$. Thus, we may write $$\mathcal{S}=\mathcal{C}(I-\mathcal{A})^{-1}$$ and conclude that $\mathcal{S}$ extends to a bounded operator on  $L^p_\sigma(bD)$ since both $\mathcal{C}$ and $(I-\mathcal{A})^{-1}$ are bounded on $L^p_\sigma(bD)$. Thus, we have established all parts of Theorem \ref{szego main detailed}.
\end{proof}

\section{The Szeg\H{o} Projection on $C^2$ domains}\label{szego minimal} 
\subsection{Background for $C^2$ case}
We now consider what modifications are necessary to prove Theorem \ref{szego main2}, as in \cite{LS1}. From now on we assume $D$ has boundary of class $C^2$, but all the other assumptions about $D$ and $\rho$ from before remain in force.  We shall be brief, as basically the same setup applies with one crucial change. This involves uniformly approximating the second derivatives of $\rho$ by differentiable functions. In particular, since that boundary is of class $C^2$, we must replace the second derivatives $\frac{\partial^2 \rho}{\partial w_j \partial w_k}$ by an $n \times n$ matrix of $\{\tau_{j,k}^\varepsilon\}$ of $C^1$ functions satisfying

$$\sup_{w \in bD} \left| \frac{\partial \rho}{\partial w_j \partial w_k}(w)-\tau_{j,k}^\varepsilon(w) \right| \leq \varepsilon \quad 1\leq j,k \leq n.$$

Now we define the analogs of $g(w,z)$, $G(w,z)$ and $\eta(w,z)$ In particular, define

$$g_\varepsilon(w,z):= \chi\left( \sum_{j=1}^{n} \dfrac{\partial \rho}{\partial w_j}(w)(w_j-z_j) -\frac{1}{2}\sum_{j,k=1}^{n} \tau_{j,k}^{\varepsilon}(w) (w_j-z_j)(w_k-z_k)\right)+(1-\chi)|w-z|^2$$ where $\chi$ is the same $C^\infty$ cutoff function as in the $C^3$ case. If $\varepsilon$ is taken sufficiently small, we have the analogous estimate

$$\text{Re}(g_\varepsilon(w,z)) \gtrsim -\rho(z)+|w-z|^2,$$ where the implicit constant is independent of $\varepsilon.$

 In the same way, we define the $(1,0)$ form in $w$ $G_\varepsilon(w,z)$ as follows:

$$G_\varepsilon(w,z):= \chi\left( \sum_{j=1}^{n} \dfrac{\partial \rho}{\partial w_j}(w)dw_j -\frac{1}{2}\sum_{j,k=1}^{n} \tau_{j,k}^{\varepsilon }(w)(w_k-z_k)dw_j\right)+(1-\chi)\sum_{j=1}^n(\bar{w}_j-\bar{z}_j)dw_j.$$

As before, we define for $w \in bD, z\in D$:

$$\eta_\varepsilon(w,z):= \frac{G_\varepsilon(w,z)}{g_\varepsilon(w,z)}.$$

Then of course $\eta_\varepsilon$ is again a generating form. Therefore, we can construct the associated Cauchy-Fantappi\'{e} integral operator $\mathbf{C}_\varepsilon^1$ in exactly the same way as we constructed $\mathbf{C}_1$, with $\eta_\varepsilon$ playing the role of $\eta$. In particular, the analog of Proposition \ref{reproduce holo} holds for $\mathbf{C}_\varepsilon^1$. 

The issue, again, is that $\mathbf{C}_\varepsilon^1$ reproduces but does not produce holomorphic functions. Again, we can introduce a correction operator $\mathbf{C}_\varepsilon^2$ and consider the operator $\mathbf{C}=\mathbf{C}_\varepsilon^1+\mathbf{C}_\varepsilon^2.$ Proposition \ref{reproduce and produce} will hold in this case; the operator $\mathbf{C}_\varepsilon$ will reproduce and produce holomorphic functions. 

The rest of the setup follows basically identically. The definition of the Leray-Levi measure $\mathop{d \lambda}$ does not change, except now $\Lambda$ will merely be a continuous rather than Lipschitz map. The quasi-metric $d$ will be defined in the same way, namely

$$d(w,z)=|g_\varepsilon(w,z)|^{1/2}$$ and will satisfy the same properties, including $(bD,d,\mu)$ being a space of homogeneous type. 

We can again consider the operator 

$$\mathcal{C}_\varepsilon(f)(z)=\mathbf{C}_\varepsilon(f)(z)\rvert_{bD}$$ and this definition makes sense when the function $f$ is H\"{o}lder continuous with respect to $d$ as before. We also can obtain the decomposition

$$\mathbf{C}_\varepsilon=\mathbf{C}_\varepsilon^\sharp+\mathbf{R}_\varepsilon$$ where $$\mathbf{C}_\varepsilon^\sharp(f)(z)=\int_{bD} \frac{f(w)}{g_\varepsilon(w,z)^n} \mathop{d \lambda(w)}$$ and the kernel $R_\varepsilon(w,z)$ of the operator $\mathbf{R}_\varepsilon$ satisfies

$$|R_\varepsilon(w,z)| \leq c_{\varepsilon} d(w,z)^{-2n+1}.$$ Here $c_\varepsilon$ denotes a constant that can depend on $\varepsilon$.

Restricting this decomposition to the boundary, it is possible to obtain the following operator equation, acting on an appropriate class of functions:

$$\mathcal{C}_\varepsilon= \mathcal{C}_\varepsilon^\sharp+\mathcal{R}_\varepsilon.$$

The class of $A_p$ weights and the maximal function are defined in the exact same manner as before.

This concludes our reiteration of the preliminaries for the $C^2$ case. The reader is invited to consult \cite{LS1} for more details.

\subsection{Weighted estimates in the $C^2$ case}

We now demonstrate how weighted $L^p$ bounds can be obtained in the $C^2$ case. Throughout we closely follow the arguments in \cite{LS1}. First, note that we can still obtain the Kerzman-Stein equation in the same way as before. Thus, we have on $L^2(bD)$:

\begin{equation}\mathcal{S}(I-(\mathcal{C}_\varepsilon^*-\mathcal{C}_\varepsilon))=\mathcal{C}_\varepsilon. \label{10}\end{equation}

In this case, we will be unable to invert the operator $(I-(\mathcal{C}_\varepsilon^*-\mathcal{C}_\varepsilon))$. It suffices to prove that $\mathcal{S}$ is bounded on $L^2_\sigma(bD)$ for all $\sigma \in A_2$; then we can appeal to extrapolation. To begin with, we have:

\begin{lemma}\label{main c2 term}
For $\sigma \in A_2$ the operator $\mathcal{C}_\varepsilon$ extends to a bounded operator on $L^2_\sigma(bD)$ and in particular satisfies
 $$||\mathcal{C}_\varepsilon f||_{L^2_\sigma(bD)} \leq c_{\varepsilon,\sigma} ||f||_{L^2_\sigma(bD)},$$ where $c_{\varepsilon,\sigma}$ is a constant that depends on $\varepsilon$ and the weight $\sigma$.

 \begin{proof}
 First, the operator $\mathcal{C}_\varepsilon^\sharp$ is Calder\'{o}n-Zygmund (see the proof of \cite[Theorem 7]{LS1}); however, the constants in its smoothness estimates do depend on $\varepsilon$. The bound on the kernel of $\mathcal{R}_\varepsilon$ in fact implies that it is compact on $L^2_\sigma(bD)$ by the arguments in Lemma \ref{compactness}. This finishes the proof.
 \end{proof} 
 \end{lemma}
 
 The dependence of the constant on $\varepsilon$ turns out not to be an issue because ultimately in the course of the proof we will fix $\varepsilon$ sufficiently small and do not need to take a limit as $\varepsilon \rightarrow 0$. 
 
Next, we need to break up the operator $\mathcal{C}_\varepsilon^*-\mathcal{C}_\varepsilon$. Roughly, we break the kernel of $\mathcal{C}_\varepsilon$ into pieces supported on and off the diagonal $w=z$. Let $s=s(\varepsilon)$ be a parameter chosen depending on $\varepsilon$.  We write

$$\mathcal{C}_\varepsilon= \mathcal{C}_\varepsilon^s+\mathcal{R}_\varepsilon^s$$where $$\mathcal{C}_\varepsilon^s(f)= \mathcal{C}_\varepsilon(f \chi_s)
$$ and $\chi_s(w,z)$ is a symmetrized smooth cutoff function that is 1 when $d(z,w)\leq cs$ and 0 when $d(z,w)\geq s$ (see \cite{LS1} for details). Thus,
$$\mathcal{C}_\varepsilon^*-\mathcal{C}_\varepsilon=[(\mathcal{C}_\varepsilon^s)^*-\mathcal{C}_\varepsilon^s]+[(\mathcal{R}_\varepsilon^s)^*-\mathcal{R}_\varepsilon^s]:=\mathcal{A}_\varepsilon+\mathcal{D}_\varepsilon.$$

It is immediate from previous discussions that for fixed $\varepsilon$, $s$, the kernel of $\mathcal{R}_\varepsilon^s$ is bounded. It is then an entirely straightforward exercise using H\"{o}lder's inequality and the integrability of $\sigma$ that $\mathcal{R}_\varepsilon^s$ boundedly maps $L^2_\sigma(bD)$ to $L^\infty(bD)$.

We now need to deal with the other term. First, we state a lemma (\cite[Lemma 24]{LS1}) that we will later need. It is a decomposition lemma that partitions $\mathbb{C}^n=\mathbb{R}^{2n}$ into cubes at various levels. In particular, let $Q^1_0$ denote the unit cube centered at the origin in $\mathbb{C}^n$, and for $k \in \mathbb{Z}^n$ let $Q^1_k=k+Q^1_0$ be its integer translates. For $\gamma>0$, let $Q^\gamma_k=\gamma Q^1_k$. Note that for a given cube $Q^\gamma_k$, there are at most $N=3^{2n}$ cubes that touch it; i.e whose closures have non-empty intersection.

\begin{lemma}\label{cube decomp} Fix $\gamma>0$. Suppose $T$ is a bounded operator on $L^2_\sigma(bD)$ that satisfies:
\begin{enumerate}
\item $\mathbbm{1}_j T \mathbbm{1}_k=0$ if the cubes $Q_j^\gamma$ and $Q_k^\gamma$ do not touch.
\item $ ||\mathbbm{1}_j T \mathbbm{1}_k||_{L^2_\sigma} \leq A$ otherwise.
\end{enumerate}

Then $T$ satisfies $$ || T||_{L^2_\sigma} \leq AN.$$

\begin{proof}
The proof is identical to the one given in \cite{LS1}. The underlying measure is now $\sigma \mathop{d \mu}$ as opposed to just Lebesgue measure, but the argument is the same.
\end{proof}
\end{lemma}

We have the following theorem:

\begin{lemma}\label{small norm operator}
Given $\varepsilon>0$, there exists an $s=s(\varepsilon)$ so the following holds:

$$||(\mathcal{C}_\varepsilon^s)^*-\mathcal{C}_\varepsilon^s||_{L^2_\sigma(bD)} \leq \varepsilon^{1/2} M_{p,\sigma}$$ where the constant $M_{p,\sigma}$ depends on $p$ and the weight $\sigma$ but not $\varepsilon$.
\begin{proof}
Here the distinction between the Leray-Levi measure and Lebesgue measure becomes important. As before, let $\dagger$ denote the adjoint of an operator taken with respect to Leray-Levi measure, and write

$$(\mathcal{C}_\varepsilon^s)^*-\mathcal{C}_\varepsilon= [(\mathcal{C}_\varepsilon^s)^\dagger-\mathcal{C}_\varepsilon]+[(\mathcal{C}_\varepsilon^s)^*-(\mathcal{C}_\varepsilon^s)^\dagger].$$ We will first show $$||(\mathcal{C}_\varepsilon^s)^\dagger-\mathcal{C}_\varepsilon^s||_{L^2(\sigma)} \leq \varepsilon^{1/2} M_{p,\sigma}.$$

Note as before we decomposed $\mathcal{C}_\varepsilon$, we can write $\mathcal{C}_\varepsilon^s=\mathcal{C}_\varepsilon^{\sharp,s}+\mathcal{R}_\varepsilon^{\sharp,s}$, where $\mathcal{C}_\varepsilon^{\sharp,s}$ is the corresponding truncation of the operator $\mathcal{C}_\varepsilon^\sharp$. Write $$(\mathcal{C}_\varepsilon^s)^\dagger-\mathcal{C}_\varepsilon= [(\mathcal{C}_\varepsilon^{\sharp,s})^\dagger-\mathcal{C}_\varepsilon^{\sharp,s}]+[(\mathcal{R}_\varepsilon^{\sharp,s})^\dagger-\mathcal{R}_\varepsilon^{\sharp,s}]=\mathcal{A}_\varepsilon^s+\mathcal{B}_\varepsilon^s.$$

Recall that the kernel of $\mathcal{R}_\varepsilon$ is majorized by $c_\varepsilon d(w,z)^{-2n+1}.$ Using basically the arguments of Proposition \ref{A_1 control}, we have, for any $\sigma' \in A_1$:

$$\mathcal{R}_\varepsilon^{\sharp,s}(\sigma')(z) \lesssim s \sigma'(z)$$ and

$$(\mathcal{R}_\varepsilon^{\sharp,s})^*(\sigma')(z) \lesssim s \sigma'(z)$$ where the implicit constants depend on the weight $\sigma'$ and $\varepsilon$. Then, by writing $\sigma \in A_2$ as a quotient of $A_1$ weights and applying the reasoning in the proof of Lemma \ref{compactness}, it is straightforward to show that $||\mathcal{R}_\varepsilon^{\sharp,s}||_{L^2_\sigma(bD)} \leq c_{\varepsilon,\sigma} s M_{p,\sigma}.$ Choosing $s$ appropriately small in terms of $\varepsilon$, we obtain the estimate $$||\mathcal{R}_\varepsilon^{\sharp,s}||_{L^2_\sigma(bD)} \leq \varepsilon^{1/2} M_{p,\sigma},$$ as desired. The same estimate is easily seen to hold for $(\mathcal{R}_\varepsilon^{\sharp,s})^\dagger,$ proving the estimate for $\mathcal{B}_\varepsilon^s$.

We now turn to $\mathcal{A}_\varepsilon^s$. It is proven in \cite{LS1} that the operators $\varepsilon^{-1/2}\mathcal{A}_\varepsilon^s$ satisfy smoothness and cancellation conditions that are uniform in $\varepsilon$. Lanzani and Stein apply the $T(1)$ theorem to show that $||\mathcal{A}_\varepsilon^s||_{L^p(bD)} \leq \varepsilon^{1/2} M_p$, where $M_p$ is independent of $\varepsilon$. But the same Calder\'{o}n-Zygmund theory shows that $$||\mathcal{A}_\varepsilon^s||_{L^2_\sigma(bD)} \leq \varepsilon^{1/2} M_{p,\sigma},$$ as we sought to show. We have thus demonstrated the result for $(\mathcal{C}_\varepsilon^s)^\dagger-\mathcal{C}_\varepsilon^s.$

We now turn to the operator $(\mathcal{C}_\varepsilon^s)^*-(\mathcal{C}_\varepsilon^s)^\dagger.$ Estimating the norm of this operator turns out to involve estimating the norm of a commutator. In particular, $(\mathcal{C}_\varepsilon^s)^*-(\mathcal{C}_\varepsilon^s)^\dagger=(\mathcal{C}_\varepsilon^s)^*-\Lambda (\mathcal{C}_\varepsilon^s)^*\Lambda^{-1},$ where $\mathop{d \lambda}=\Lambda \mathop{d \mu}$ and $\Lambda$ is a continuous function that is bounded above and below. Thus, the $L^2_\sigma$ norm of this operator is controlled by

$$||\Lambda||_{L^\infty(bD)} \left\Vert\left[\Lambda^{-1},(\mathcal{C}_\varepsilon^s)^*\right]\right\Vert_{L^2_\sigma(bD)},$$ where $[A,B]=AB-BA$.

Notice by a simple computation, $$\left(\left[\Lambda^{-1},(\mathcal{C}_\varepsilon^s)^*\right]\right)^*=\mathcal{C}_\varepsilon^s \bar{\Lambda}^{-1}-\bar{\Lambda}^{-1}\mathcal{C}_\varepsilon^s=\left[\mathcal{C}_\varepsilon^s, \bar{\Lambda}^{-1}\right],$$ so by duality it suffices to estimate the norm of a commutator $\left[\mathcal{C}_\varepsilon^s,\phi \right]$ on $L^2_\sigma(bD)$ for any $\sigma \in A_2$, where $\phi$ is an arbitrary continuous map $bD \rightarrow \mathbb{C}$. In particular, we claim for fixed $\phi$: 
$$\left \Vert \left[\mathcal{C}_\varepsilon^s,\phi \right]\right \Vert_{L^2_\sigma(bD)} \leq \varepsilon M_{p,\sigma}.$$ This is exactly proven in \cite{LS1}, but for unweighted $L^p$. A key ingredient in the proof is contained in \cite[Proposition 19]{LS1}, which states that we can get a uniform bound $||\mathcal{C}_\varepsilon^s||_{L^p(bD)} \leq M_p$ for $\varepsilon$ and $s$ chosen sufficiently small. This is proven using the $T(1)$ theorem with estimates uniform in $\varepsilon$, but then of course the same proof implies

$$||\mathcal{C}_\varepsilon^s||_{L^2_\sigma(bD)} \leq M_{p,\sigma}.$$

Now we provide a short sketch of how Lemma \ref{cube decomp} leads to the desired conclusion again following the arguments from \cite{LS1}. In particular, we apply the lemma to the operator $[\mathcal{C}_\varepsilon^s,\phi]$ with $\varepsilon$ and $s$ chosen appropriately. The first condition of Lemma \ref{cube decomp} basically follows because $\mathcal{C}_\varepsilon^s$ has a kernel that is supported in a small neighborhood of the diagonal (in particular, we take $\gamma=cs$). 

The second condition follows from the (uniform) continuity of $\phi$. For a cube $Q_k^\gamma$, denote its center by $z_k$. If $s$ is chosen sufficiently small, then by continuity, if $z \in Q_j^\gamma$, where $Q_j^\gamma$ touches $Q_k,$ we have

$$|\phi(z)-\phi(z_k)|< \varepsilon.$$

Now write $\phi=\phi_k+\psi_k$, where $\phi_k(z)=\phi(z)-\phi(z_k)$ and $\psi_k(z)=z_k$. Obviously, $[\mathcal{C}_\varepsilon^s,\phi]=[\mathcal{C}_\varepsilon^s,\phi_k]+ [\mathcal{C}_\varepsilon^s,\psi_k],$ but $[\mathcal{C}_\varepsilon^s,\psi_k]=0$ as $\psi_k$ is constant. Therefore, we have for any cube $Q_j^\gamma$ that touches $Q_k^\gamma$:
\begin{eqnarray*}
||\mathbbm{1}_j[\mathcal{C}_\varepsilon^s,\phi]\mathbbm{1}_k||_{L^2_\sigma(bD)} & = & ||\mathbbm{1}_j[\mathcal{C}_\varepsilon^s,\phi_k]\mathbbm{1}_k||_{L^2_\sigma(bD)}\\
& \leq & ||\mathbbm{1}_j\mathcal{C}_\varepsilon^s \phi_k \mathbbm{1}_k||_{L^2_\sigma(bD)}+ ||\mathbbm{1}_j\phi_k \mathcal{C}_\varepsilon^s \mathbbm{1}_k||_{L^2_\sigma(bD)}\\
& < & 2 \varepsilon ||\mathcal{C}_\varepsilon^s||_{L^2_{\sigma}(bD)}\\
& \leq & 2 \varepsilon M_{p,\sigma}.
\end{eqnarray*}

This completes the proof.

\end{proof}
\end{lemma}

The following proposition is an immediate consequence of the well-known reverse H\"{o}lder property of $A_p$ weights. 
\begin{proposition}\label{reverse Holder}
Let $1<p<\infty$ and suppose $\sigma \in A_p$. Then there exists a $\delta>0$ so $\sigma^{1+\delta} \in L^1(bD)$.
\end{proposition}

We are now finally ready to prove the main theorem. 

\begin{proof}[Proof of Theorem \ref{szego main2}]
As noted before, it suffices to prove the result for $p=2.$ Recall $\mathcal{A}_\varepsilon=(\mathcal{C}_\varepsilon^s)^*-\mathcal{C}_\varepsilon^s$ and $\mathcal{D}_\varepsilon=(\mathcal{R}_\varepsilon^s)^*-\mathcal{R}_\varepsilon^s$.  Thus, the Kerzman-Stein equation takes the form
$$\mathcal{S}(I-\mathcal{A}_\varepsilon)-\mathcal{S}\mathcal{D}_\varepsilon=\mathcal{C}_\varepsilon.$$ By Lemma \ref{small norm operator}, if $\varepsilon$ and $s$ are chosen sufficiently small, then $||\mathcal{A}_\varepsilon||_{L^2_\sigma(bD)}<1.$ Inverting $\mathcal{A}_\varepsilon$ using a Neumann series yields:

$$\mathcal{S}=\mathcal{C}_\varepsilon(I-\mathcal{A}_\varepsilon)^{-1}+\mathcal{S}\mathcal{D}_\varepsilon(I-\mathcal{A}_\varepsilon)^{-1}.$$

By Lemma \ref{main c2 term}, the operator $\mathcal{C}_\varepsilon(I-\mathcal{A}_\varepsilon)^{-1}$ maps $L^2_\sigma(bD)$ to itself. Now, by discussions above $\mathcal{D}_\varepsilon(I-\mathcal{A}_\varepsilon)^{-1}$ maps $L^2_\sigma(bD)$ to $L^\infty(bD)$, and hence maps $L^2_\sigma(bD)$ to $L^p(bD)$ boundedly for any $p$, $1<p<\infty$. Additionally, by the principle result in \cite{LS1}, $\mathcal{S}$ extends to a bounded operator on $L^p$. So in particular $\mathcal{S}\mathcal{D}_\varepsilon(I-\mathcal{A}_\varepsilon)^{-1}$ maps $L^2_\sigma(bD)$ to $L^p$ for all $p$, $1<p<\infty$. We claim that if $p$ is chosen sufficiently large (depending on $\sigma$), then $||g||_{L^2_\sigma(bD)} \lesssim ||g||_{L^p(bD)}$ for all measurable functions $g$. Then $$||\mathcal{S}\mathcal{D}_\varepsilon(I-\mathcal{A}_\varepsilon(f))||_{L^2_\sigma(bD)} \lesssim ||\mathcal{S}\mathcal{D}_\varepsilon(I-\mathcal{A}_\varepsilon(f))||_{L^p(bD)}$$ for all measurable $f$, which will then establish the result.

To prove the claim, we use Proposition \ref{reverse Holder}. In particular, we have, using H\"{o}lder's inequality with exponents $\frac{p}{2}$ and $r=\left(\frac{p}{2}\right)'$:
\begin{eqnarray*}
||g||_{L^2_\sigma(bD)}^2 & = & \int_{bD} |g|^2 \sigma \mathop{d \mu}\\
& \leq & \left(\int_{bD} |g|^p \mathop{d \mu}\right)^{\frac{2}{p}} \left(\int_{bD} \sigma^{r} \mathop{d \mu}\right)^{\frac{1}{r}}\\
& \lesssim & ||g||_{L^p(bD)}^2
\end{eqnarray*}

 \noindent provided $p$ is chosen so $r<1+\delta$. This completes the proof.

\end{proof}

\section{The Bergman Projection}\label{Bergman section}
In this section, we will study the Bergman projection on weighted spaces under the assumption that $D$ is a $C^4$ domain. Our main goal is to prove the following theorem, which is a more detailed version of Theorem \ref{bergman main}. Here $\mathcal{T}$ is the auxiliary operator corresponding to the Bergman projection that we discussed in subsection \ref{Intro Proof Outline}.

\begin{theorem}\label{bergman main detailed}
Let $D$ be strongly pseudoconvex with $C^4$ boundary. Then for $1<p<\infty$ and $\sigma \in B_p$, the following hold:
\begin{enumerate}
\item The operator $\mathcal{T}^*-\mathcal{T}$ is compact on $L^p_\sigma(D)$.
\item The operator $I-(\mathcal{T}^*-\mathcal{T})$ is invertible on $L^p_\sigma(D)$.
\item The Bergman projection $\mathcal{B}$ extends to a bounded operator on $L^p_\sigma(D)$ and satisfies

$$\mathcal{B}=\mathcal{T}(I-(\mathcal{T}^*-\mathcal{T}))^{-1}.$$

\end{enumerate}
\end{theorem}

\subsection{Background and Setup}\label{theproblemBergman}
Now we let $D$ be a strongly pseudoconvex  domain with $C^4$ defining function $\rho$. As in Lanzani-Stein \cite{LS2}, we can construct an integral operator $\mathcal{T}=\mathcal{T}_1+\mathcal{T}_2$ that integrates over the interior of the domain $D$, where $\mathcal{T}_1$ is constructed using Cauchy-Fantappi\'{e} theory and $\mathcal{T}_2$ is obtained by solving a $\bar{\partial}$ problem. The operator $\mathcal{T}$ has the property that it produces and reproduces holomorphic functions.  

We now make several definitions that are analogous to our treatment above of the Szeg\H{o} projection. We will slightly abuse notation by reusing certain letters to represent analogous objects in the Bergman case.
Define

$$g(w,z):=-\rho(w)-\chi(P_w(z))+(1-\chi)|w-z|^2$$

\noindent where $P_w(z)$ denotes the Levi polynomial at $w$ and $\chi$ is an appropriately chosen $C^\infty$ cutoff function. In particular, using the strict pseudoconvexity of $D$, $\chi$ can be chosen so 

$$\text{Re } g(w,z) \gtrsim
-\rho(w)-\rho(z)+c|w-z|^2 .$$

Now, as before define the $(1,0)$ form in $w$  

$$G(w,z):= \chi\left(\sum_{j=1}^{n} \frac{\partial \rho}{\partial w_j}(w)\mathop{d w_j} -\frac{1}{2}\sum_{j,k=1}^n\frac{\partial^2 \rho}{\partial w_j \partial w_k }(w)(w_k-z_k)\mathop{d w_j}\right)+ (1-\chi) \sum_{j=1}^n (\bar{w}_j-\bar{z}_j)\mathop{d w_j}.$$

\noindent Note that $G$ has the property that if we let 

\begin{equation}\hat{\eta}(w,z)=\frac{G(w,z)}{g(w,z)+\rho(w)}, \label{11}\end{equation}

\noindent then $$\langle \hat{\eta}(w,z),w-z\rangle= 1$$
\noindent for all $z \in D$ and $w$ in neighborhood of $bD$. Note that \eqref{11} indicates $\hat{\eta}$ is a generating form. However, we instead define the $(1,0)$ form in $w$:

$$\eta(w,z)=\frac{G(w,z)}{g(w,z)}$$

\noindent and associated integral operator 

$$\mathcal{T}_1(f)(z):=\frac{1}{(2 \pi \ii)^n} \int_D (\bar{\partial_w}\eta)^n(w,z) f(w),$$ where $(\bar{\partial_w}\eta)^n$ denotes the wedge product taken $n$ times. We have the following proposition (see \cite[Proposition 3.1]{LS2}):

\begin{proposition}\label{reproduce interior} Suppose $f$ is holomorphic on $D$ and belongs to $L^1(D)$. Then for all $z \in D$, one has

$$\mathcal{T}_1(f)(z)=f(z).$$

\end{proposition}

A computation shows the operator $\mathcal{T}_1$ has kernel \begin{equation} K_1(w,z)=\frac{N(w,z)}{(g(w,z))^{n+1}} \label{12} \end{equation} where $N(w,z)$ is an $(n,n)$ form of class $C^1$ (in $w$) with coefficients smooth in $z$. In particular, we have (see \cite{LS2}):
\begin{equation}N(w,z)= -\left((\bar\partial \eta)^{n-1} \wedge \bar\partial_w g \wedge \eta+g(\bar \partial \eta)^n \right).\label{13} \end{equation} We write $N(w,z)=\mathcal{N}(w,z)dV(w)$, where $dV$ denotes the Euclidean volume form. Notice the fact that $\mathcal{N}(w,z)$ is of class $C^1$ in $w$ is a direct consequence of the fact that $D$ has $C^4$ boundary.

Proposition \ref{reproduce interior} guarantees that $\mathcal{T}_1$ reproduces holomorphic functions, but as in the Szeg\H{o} case we need to add a correction operator to ensure that it produces holomorphic functions. The details can be found in \cite{LS2}, and again involve solving a $\bar{\partial}$ problem on a strongly pseudoconvex domain that contains $D$. We have the following proposition concerning $\mathcal{T}_2$ and the operator $\mathcal{T}=\mathcal{T}_1+\mathcal{T}_2$ (see \cite[Proposition 3.2]{LS2}):

\begin{proposition} \label{correction bergman}
There is an integral operator  $\mathcal{T}_2$ defined
$$\mathcal{T}_2f(z):= \int_{D} K_2(z,w) f(w) \mathop{d V(w)}$$ with 

$$\sup_{(z,w) \in \bar{D} \times \bar{D}}|K(z,w)|<\infty$$  that satisfies:
\begin{enumerate}
\item If $f \in L^1(D)$, then $\mathcal{T}(f)$ is holomorphic on $D$.

\item If, in addition, $f$ is holomorphic on $D$, then $\mathcal{T}(f)(z)=f(z)$ for $z \in D$.

\end{enumerate}
\end{proposition}

We now introduce an appropriate quasi-metric which gives rise to a space of  homogeneous type on $D$. This metric can be defined using polydiscs introduced by McNeal (see \cite{Mc2}) and is defined locally at first on a neighborhood $U$ of a point $p \in bD$. Fix a point $w \in U$. First, we may by a unitary rotation (plus a normalization) and translation assume $\partial \rho(w)=d z_1$ and $w=0$. Then, define holomorphic coordinates $\zeta= (\zeta_1,\dots,\zeta_n)$ as follows:
$$\zeta_1=z_1+\frac{1}{2}\sum_{j,k=1}^{n}\frac{\partial^2 \rho(w)}{\partial z_j \partial z_k}(z_j)(z_k), \hspace{0.2 cm} \zeta_j=z_j, j=2,\dots n  .$$ Note if $\Phi: U \rightarrow \Phi(U)$ denotes this coordinate map, $\Phi$ is a biholomorphism if $U$ is chosen small enough.

Consider the polydisc:
$$P(w,\delta)=\{z: |z_1|<\delta, |z_j|<\delta^{1/2}, 2 \leq j \leq n\},$$ 
where again $z_j$ denotes the special holomorphic coordinates centered at $w.$ 

These polydiscs satisfy certain types of doubling properties (see \cite{Mc1}). We include a proof for completeness.
\begin{proposition}
There exist independent constants $C_1,C_2$ so the following hold for the polydiscs:
\begin{enumerate}
\item If $P(q_1,\delta)\cap P(q_2, \delta) \neq \emptyset$, then $P(q_1,\delta) \subset C_1 P(q_2,\delta)$ and $P(q_2,\delta)\subset C_1 P(q_1,\delta)$.

\item There holds $P(q_1,2 \delta) \subset C_2 P(q_1, \delta).$
\end{enumerate}

\begin{proof}
The second property is essentially immediate from the definition of $P$, so we focus on the first property. Suppose $P(q_1,\delta)\cap P(q_2, \delta) \neq \emptyset$. Let $z_1,\dots,z_n$ denote the holomorphic coordinates centered at $q_1$ and $\zeta_1,\dots,\zeta_n$ denote the holomorphic coordinates centered at $q_2.$ The general idea is that these holomorphic coordinates do not differ greatly. We need to take an arbitrary point $p \in P(q_1,\delta)$ and show there exists a constant $C_1$ so $p \in C_1P(q_2,\delta)$. Let $r \in  P(q_1,\delta)\cap P(q_2, \delta)$. Write the coordinates of $p$ relative to the coordinate system of the second polydisc as $(\zeta_1(p),\dots,\zeta_n(p))$. First observe that the definition of the polydiscs implies $$|p-q_2| \leq |p-r|+|r-q_2| \lesssim \delta^{1/2}$$ and the same bound holds for the quantities $|q_1-q_2|$ and $|p-q_1|$. Then we have 
\begin{eqnarray*}
|\zeta_1(p)| & \approx & \left| \sum_{j=1}^{n} \frac{\partial \rho}{\partial z_j }(q_2)(p_j-q_{2,j})\right|+\mathcal{O}(|p-q_2|^2)\\
& \lesssim & |z_1(p)|+ \left| \sum_{j=1}^{n} \frac{\partial \rho}{\partial z_j }(q_2)(p_j-q_{2,j})-\sum_{j=1}^{n} \frac{\partial \rho}{\partial z_j }(q_1)(p_j-q_{1,j})\right|+\delta\\
& \lesssim & \delta + \left|\langle \partial \rho(q_2)-\partial \rho(q_1), p-q_2 \rangle \right|+ \left|\langle \partial \rho(q_1), q_2-q_1\rangle \right|\\
& \lesssim & \delta+|q_2-q_1||p-q_2|+ \left|\langle \partial \rho(q_1), q_2-q_1\rangle\right|\\
& \lesssim & \delta+ \left|\langle \partial \rho(q_1), q_2-q_1\rangle\right|.
\end{eqnarray*}

We control $\left|\langle \partial \rho(q_1), q_2-q_1\rangle\right|$ as follows:
\begin{eqnarray*}
\left|\langle \partial \rho(q_1), q_2-q_1\rangle\right| & \leq & \left| \langle \partial \rho(q_1),r-q_1 \rangle\right|+ \left|\langle \partial \rho(q_1), q_2-r\rangle\right|\\
& \leq & z_1(r)+ \left| \langle \partial \rho(q_1)-\partial \rho(q_2), q_2-r \rangle \right| + \left| \langle \partial  \rho(q_2), r-q_2 \rangle \right| \\
& \lesssim & \delta+|q_1-q_2||q_2-r|+\zeta_1(r)\\
& \lesssim & \delta.
\end{eqnarray*}

It is easy to verify all the implicit constants are independent of $q_1,q_2$. So there exists a constant $C_1$ so $|\zeta_1(p)| < C_1 \delta$. 

On the other hand, for $2 \leq j \leq n$, we have 

$$|\zeta_j(p)| \lesssim |p-q_2| \lesssim \delta^{1/2},$$ so if $C_1$ is chosen appropriately large, then $|\zeta_j(p)|< C_1 \delta^{1/2}$. Then $p \in C_1 P(q_1, \delta)$, as we sought to show.

The other conclusion is immediate by symmetry. This completes the proof. 
\end{proof}
\end{proposition}

As a consequence of these covering properties, one can now introduce a local quasi-metric $M$ on $U$:

\begin{definition} Define the following function on $U \times U$:
$$M(z,w)= \inf_{\varepsilon>0}\{\varepsilon: w \in P(z, \varepsilon)\}.$$
Then $M$ defines a quasi-metric on $U$. The argument is essentially the same as the other cases considered in \cite{Mc1}.
\end{definition}

It is also routine to verify that $M(z,w)$ is comparable to the following metric quantity:

$$M(z,w) \approx |z_1-w_1|+\sum_{j=2}^{n}|z_j-w_j|^2$$

\noindent where again the components of $z$ and $w$ are computed in the special coordinates centered at $w$.

It is possible to patch together these local quasi-metrics together to obtain a global quasi-metric $d(z,w)$ that is comparable to each local piece (again the argument is essentially contained in \cite{Mc1}). Technically, this metric is only defined on a tubular neighborhood of the boundary, but this presents us with no issues and we abuse notation by writing it to be defined on $D$ (see, for example \cite{HWW}). 

It follows that $(D,d,\mathop{dV})$ is a space of homogeneous type in the sense of Coifman and Weiss, where $\mathop{dV}$ denotes Lebesgue measure on $D$. We may symmetrize $d$ by taking replacing it with $d(z,w)+d(w,z)$ and assume $d(z,w)=d(w,z)$. It is also a fact that $V(B(z,r))\approx r^{n+1},$ where $B(z,r)=\{w \in D: d(w,z)<r\}$ (note the biholomorphism is measure-preserving, see also \cite{HWW, HWW2}). Moreover, we can define the distance to the boundary in this metric:

$$d(z,bD):= \inf_{w \in bD}d(z,w).$$

\noindent It is verified in \cite{HWW} that this quantity is comparable to the Euclidean distance to the boundary.

We have the following relation between the quasi-metric $d$ and the Euclidean distance:

\begin{proposition}\label{euclidean bounds}
We have, for $z',z \in D$:
$$|z-z'|^2 \lesssim d(z,z') \lesssim |z-z'|.$$

\begin{proof}
It suffices to work locally, so we may assume $d$ coincides with one of the local quasi-metrics on a neighborhood $U$. Let $\Phi(z)=\zeta(z)=(\zeta_1,\dots,\zeta_n)$ denote the biholomorphic coordinate change described in detail above in the construction of $d$. Because the coordinate change is biholomorphic, we have the following bounds:
\begin{eqnarray*}
|z-z'|^2 & = & \sum_{j=1}^{n} |z_j-z'_j|^2\\
& \lesssim & \sum_{j=1}^{n} |\zeta_j-\zeta'_j|^2\\
& \leq & d(z,z').
\end{eqnarray*}
The proof of the upper bound is similar.
\end{proof}

\end{proposition}

It should also be noted that the metric $d$ extends to $\bar{D} \times \bar{D}$. We now show that when we restrict $d$ to $bD \times bD$, we obtain a quantity comparable in size to $|g(w,z)|$, which establishes a natural connection between the Szeg\H{o} and Bergman cases. 

\begin{proposition}\label{equivalent metric}
If $z,w \in bD$, then we have
$$d(z,w)\approx |g(w,z)|.$$
\begin{proof}

Let $z=(\zeta_1,\dots,\zeta_n)$ in the special holomorphic coordinates centered at $w$. Note 

$$d(z,w) \approx |\zeta_1|+\sum_{j=2}^{n}|\zeta_j|^2.$$ Also, we have by \cite[Proposition 2.1]{LS2},$$|g(w,z)| \approx |\text{Im} \langle \partial \rho(w), w-z \rangle|+|w-z|^2.$$ But notice that
$$|\langle \partial \rho(w), w-z \rangle| \lesssim |\zeta_1|+ |w-z|^2$$ and moreover

$$|w-z|^2 \lesssim \sum_{j=1}^{n}|\zeta_j|^2 \lesssim d(z,w)$$ since the coordinate change is biholomorphic. This shows $|g(w,z)| \lesssim d(z,w)$. To see the reverse, note that if $|z-w|$ is small enough, then $g(w,z)= P_w(z)$ and

$$|\zeta_1| \lesssim |P_w(z)|+|w-z|^2$$ which combined with the estimates above gives $d(z,w) \lesssim |g(w,z)|$.

\end{proof}

\end{proposition}

We can now define a suitable class of $B_p$ weights on the domain $D$. Loosely speaking, this condition imposes that the product of the average of $\sigma$ and the average of $\sigma^\frac{-1}{p-1}$ is controlled on quasi-balls that touch the boundary of $D$ (or so-called Carleson tents). In what follows, let $\sigma$ be a locally integrable function that is positive almost everywhere. 

\begin{definition}
For $1<p<\infty$, we say the weight $\sigma$ belongs to the B\'{e}koll\`{e}-Bonami ($B_p$) class associated to the quasi-metric $d$ if the following quantity is finite:

$$[\sigma]_{B_p}:= \sup_{B(w,R); R>d(w,bD)} \left(\frac{1}{V(B(w,R))}\int_{B(w,R)} \sigma \mathop{dV}\right)\left(\frac{1}{V(B(w,R))}\int_{B(w,R)} \sigma^{-1/(p-1)} \mathop{dV}\right)^{p-1}.$$

\end{definition}

We can also define an associated maximal function:

\begin{definition}\label{maximal}
For  $z \in D$ and $f \in L^{1}(D)$, define the following maximal function:
$$\mathcal{M}f(z):=\sup_{B(w,R)\ni z; R>d(w,bD)}\frac{1}{V(B(w,R))} \int_{B(w,R)}|f| \mathop{dV}.$$
\end{definition} 
\noindent It is proven in \cite{HWW} that $\mathcal{M}$ is bounded on $L^p_\sigma(D)$ for $\sigma \in B_p$.

Moreover, we can define a suitable class of $B_1$ weights (again $\sigma$ is a locally integrable function on $D$ that is positive almost everywhere). 

\begin{definition} We say the weight $\sigma$ belongs to the class $B_1$ if for all $z \in D$, 
$$\mathcal{M}(\sigma)(z) \lesssim \sigma(z).$$
\end{definition}

\subsection{The Main Term}

We follow the following general outline to prove Theorem \ref{bergman main detailed}. First, we obtain size and smoothness estimates for $K_1(z,w)$, the kernel of $\mathcal{T}_1$. This enables us to prove that $\mathcal{T}$ maps $L^p_\sigma(D)$ to $L^p_\sigma(D)$. We then proceed to show that $\mathcal{T}^*-\mathcal{T}$ is compact on $L^2_\sigma(D)$ and improves $L^p$ spaces. These properties allow us to use the Kerzman-Stein equation to extract the $L^p_\sigma(D)$ boundedness of $\mathcal{B}$ from the $L^p_\sigma(D)$ boundedness of $\mathcal{T}$. 

The following proposition follows immediately from the fact that $\mathcal{T}_2$ has a bounded kernel and $D$ is a bounded domain.

\begin{proposition}\label{bounded kernel weights}
For $\sigma \in B_p$, the operator $\mathcal{T}_2$ maps $L^p_\sigma(D)$ to $L^p_\sigma(D)$ boundedly, $1<p<\infty$.

\begin{proof}
Take $f \in L^p_\sigma(D)$. Then we have
\begin{eqnarray*}
||\mathcal{T}_2(f)(z)||_{L^p_\sigma(D)}^p & = & \int_{D} \left| \int_{D} K_2(z,w) f(w) \mathop{d V(w)}\right|^p \sigma(z) \mathop{d V(z)}\\
& \lesssim & \left( \int_{D} |f(w)| \mathop{d V(w)}\right)^p \left(\int_{D}  \sigma(z) \mathop{d V(z)}\right)\\
& \leq & ||f||^p_{L^p_\sigma(D)} \left(\int_{D} \sigma(z) \mathop{d V(z)}\right)\left(\int_{D} \sigma(w)^{-\frac{1}{p-1}} \mathop{d V(w)}\right)^{p-1}\\
& \leq  & [\sigma]_{B_p}||f||^p_{L^p_\sigma(D)}.
\end{eqnarray*}

\end{proof}

\end{proposition}

We now work to prove the following theorem:

\begin{theorem}\label{aux operator} For $\sigma \in B_p$, the operator $\mathcal{T}$, as well as its adjoint $\mathcal{T}^*$, map  $L^p_\sigma(D)$ to $L^p_\sigma(D)$ boundedly, $1<p<\infty$.
\end{theorem}

In light of the previous proposition, which clearly also works for $\mathcal{T}_2^*$, it is sufficient to show that $\mathcal{T}_1$ and $\mathcal{T}_1^*$ are bounded on $L^p_\sigma(D)$. To this end, we define the following comparison operator:

$$\Gamma(f)(z)= \int_{D} \frac{1}{|g(w,z)|^{n+1}} f(w) \mathop{d V(w)}.$$ Note that in light of \eqref{12}, we have the pointwise domination:
$$|\mathcal{T}_1(f)(z)| \lesssim \Gamma(|f|)(z).$$ To prove the weighted $L^p$ regularity of $\Gamma$, we follow B\'{e}koll\`{e}'s approach of using singular integral theory that was also undertaken in \cite{HWW}. In particular, we obtain the following size and smoothness estimates on the kernel of $\Gamma$:

\begin{lemma}\label{size and smoothness} The following hold:
\begin{enumerate}
\item $$\frac{1}{|g(w,z)|^{n+1}} \lesssim \min\left\{\frac{1}{V(B(z,d(z,bD)))},\frac{1}{V(B(w,d(w,bD)))}\right\}.$$
\item If $d(z,w)\geq c d(z,z')$ for an appropriately chosen constant $c$, then 
$$\left|\frac{1}{(g(w,z))^{n+1}}-\frac{1}{(g(w,z'))^{n+1}}\right| \lesssim \left(\frac{d(z,z')}{d(z,w)}\right)^{1/2}\frac{1}{V(B(z,d(z,w)))}.$$

\item If $d(z,w)\geq c d(w,w')$ for an appropriately chosen constant $c$, then 
$$\left|\frac{1}{(g(w,z))^{n+1}}-\frac{1}{(g(w',z))^{n+1}}\right| \lesssim \left(\frac{d(w,w')}{d(z,w)}\right)^{1/2}\frac{1}{V(B(w,d(z,w)))}.$$

\end{enumerate}

\begin{proof}

For the first statement, it suffices to prove $$\frac{1}{|g(w,z)|^{n+1}} \lesssim \frac{1}{V(B(z,d(z,bD)))},$$ since $|g(w,z)| \approx |g(z,w)|$ by \cite[Proposition 2.1]{LS2}. Since  
$V(B(z,d(z,bD))) \approx [d(z,bD)]^{n+1}$, it is enough to show $d(z,bD) \lesssim |g(w,z)|$. We have $d(z,bD) \approx \operatorname{dist}(z,bD) \approx |\rho(z)|$, where $\operatorname{dist}$ denotes Euclidean distance. On the other hand, $|g(w,z)| \gtrsim |\rho(z)|$ by \cite[Proposition 2.1]{LS2}). This proves the size estimate. 

For the smoothness estimate, we first prove as a preliminary fact that $d(z,w) \lesssim |g(w,z)|$. We may assume $|w-z|$ is small enough so that $g(w,z)=-\rho(w)-P_w(z)$. By definition we have

$$d(z,w) \approx |\zeta_1|+ \sum_{j=2}^{n} |\zeta_j|^2$$ where $\zeta_1,\dots,\zeta_n$ are the components of $z$ in the holomorphic coordinates centered at $w$. Using the triangle inequality and the definition of the biholomorphic coordinates, we obtain

$$|\zeta_1| \lesssim \left|\sum_{j=1}^{n} \frac{\partial \rho (w)}{\partial z_j}(z_j-w_j)\right|+\mathcal{O}(|z-w|^2) \lesssim  |\rho(w)|+|-\rho(w)-P_w(z)|+\mathcal{O}(|z-w|^2).$$

Then, appeal to the fact that $|g(w,z)| \gtrsim |\rho(w)|+|w-z|^2$ by \cite[Proposition 2.1]{LS2} and the fact that the coordinate change is biholomorphic to obtain the desired conclusion.

We only prove the first smoothness estimate; the second one is proven similarly and is only slightly more complicated. We use similar ideas as in \cite{LS1}. We first prove the estimate

$$|g(w,z)-g(w,z')|\lesssim d(z,z')^{1/2}d(z,w)^{1,2}+d(z,z').$$

To begin with, note that we have

\begin{eqnarray*}
|g(w,z)-g(w,z')| & \leq & \left| \langle \partial \rho(w),w-z \rangle- \langle \partial \rho(w),w-z' \rangle \right|\\
 & + & \frac{1}{2} \left|\sum_{j,k=1}^{n}\frac{\partial^2 \rho(w)}{\partial w_j \partial w_k}\left[(w_j-z_j)(w_k-z_k)-(w_j-z'_j)(w_k-z'_k)\right]\right|.
\end{eqnarray*}

We deal with the first term, $\left| \langle \partial \rho(w),w-z \rangle- \langle \partial \rho(w),w-z' \rangle \right|=\left| \langle \partial \rho(w),z'-z \rangle \right|.$ We then have, using Proposition \ref{euclidean bounds}:
\begin{eqnarray*}
\left| \langle \partial \rho(w),z'-z \rangle \right| & \leq & \left| \langle \partial \rho(z),z'-z \rangle \right| + \left| \langle \partial \rho(w)-\partial\rho(z),z'-z \rangle \right| \\
& \lesssim & d(z,z')+|z-w||z-z'|\\
& \lesssim & d(z,z')+d(z,w)^{1/2}d(z,z')^{1/2}.
\end{eqnarray*}

Now we handle the second term.  Notice that we have
\begin{small}
\begin{eqnarray*}
|(w_j-z_j)(w_k-z_k)-(w_j-z'_j)(w_k-z'_k)|
& \leq & |(w_j-z_j)(w_k-z_k)-(w_j-z'_j)(w_k-z_k)|\\
& + &  |(w_j-z'_j)(w_k-z_k)-(w_j-z'_j)(w_k-z'_k)|\\
 & \leq & |w_k-z_k||z_j-z'_j|+|w_j-z'_j||z_k-z'_k|\\
& \leq & |w-z||z-z'|+(|w-z|+|z-z'|)|z-z'|\\
& \lesssim & d(z,w)^{1/2}d(z,z')^{1/2}+(d(z,w)^{1/2}+d(z,z')^{1/2})d(z,z')^{1/2}\\ 
& \lesssim & d(z,w)^{1/2}d(z,z')^{1/2}
\end{eqnarray*}
\end{small}
which proves the required bound for the second piece. 

Now, we show $|g(w,z)|\approx |g(w,z')|$ if $d(z,w) \geq c d(z,z')$. We estimate, using the work previously done:
\begin{eqnarray*}
|g(w,z)| & \leq &  |g(w,z)|+|g(w,z')-g(w,z)|\\
& \lesssim & |g(w,z')|+ d(z,w)^{1/2}d(z,z')^{1/2}+ d(z,z')\\
& \lesssim & |g(w,z')|+(c^{-1/2}+c^{-1})d(z,w)\\
& \lesssim & |g(w,z')|+ (c^{-1/2}+c^{-1})|g(w,z)|.
\end{eqnarray*}

Thus, if $c$ is chosen appropriately large, we can subtract the $|g(w,z)|$ term to the other side and obtain $|g(w,z)| \lesssim |g(w,z')|$. The bound $|g(w,z')| \lesssim |g(w,z)|$ is obtained similarly. 

Finally, we obtain, using our assumption $d(z,w)\geq c d(z,z')$:

\begin{eqnarray*} 
\left|\frac{1}{(g(w,z))^{n+1}}-\frac{1}{(g(w,z'))^{n+1}}\right| & \leq & \frac{|g(w,z)-g(w,z')|\left(\sum_{t=0}^{n}|g(w,z)|^t|g(w,z')|^{n-t}\right)}{|g(w,z)|^{n+1}|g(w,z')|^{n+1}}\\
& \lesssim & \frac{|g(w,z)-g(w,z')|}{|g(w,z)|^{n+2}}\\
& \lesssim & \frac{1}{d(z,w)^{n+1}}\frac{d(z,w)^{1/2}d(z,z')^{1/2}}{d(z,w)}\\
& \lesssim & \left(\frac{d(z,z')}{d(z,w)}\right)^{1/2}\frac{1}{V(B(z,d(z,w)))}
\end{eqnarray*} which establishes the smoothness estimate. 
\end{proof}

\end{lemma}

As a consequence of the size and smoothness estimates obtained on the kernel of the positive operator $\Gamma$, we get the following theorem (one can follow the arguments verbatim contained in \cite[Theorem 1.2]{HWW}):

\begin{theorem}\label{Gamma boundedness} For $1<p<\infty$, the operators $\Gamma, \Gamma^*$ map $L^p_\sigma(D)$ to $L^p_\sigma(D)$ boundedly for $\sigma \in B_p$. 
\end{theorem}

Now we can prove Theorem \ref{aux operator} as follows:
\begin{proof}[Proof of Theorem \ref{aux operator}] Note that Theorem \ref{Gamma boundedness} implies the operators $\mathcal{T}_1, \mathcal{T}_1^*$ map $L^p_\sigma(D)$ to $L^p_\sigma(D)$ boundedly, which together with Proposition \ref{bounded kernel weights} establishes the result.
\end{proof}

\subsection{The Error Term}

We now proceed to deal with the ``error term" $\mathcal{T}^*-\mathcal{T}$. In light of the arguments above, we already know $\mathcal{T}^*-\mathcal{T}$ is bounded on $L^p_\sigma(D)$, but in fact this operator exhibits much better behavior.  In analogy with the approach taken in this paper for the Szeg\H{o} operator, we show that this operator is compact on $L^2_\sigma(D)$ for $\sigma \in B_2$ and improves $L^p$ spaces. We conclude by applying the Kerzman-Stein trick to deduce the boundedness of $\mathcal{B}$ from this information. 

\begin{lemma}\label{smoothing kernel size} Let $K(z,w)$ denote the kernel of the integral operator $\mathcal{T}^*-\mathcal{T}$. Then we have the size estimates:

$$|K(z,w)| \lesssim d(z,w)^{-(n+\frac{1}{2})}$$ and

$$|K(z,w)| \lesssim \min\left\{ d(z,bD)^{-(n+\frac{1}{2})},d(w,bD)^{-(n+\frac{1}{2})}\right\}.$$

\begin{proof}This is where the hypothesis that $D$ has $C^4$ boundary is of importance. It is proven in \cite[Theorem 7.6]{R} that  $|K(z,w)| \lesssim |g(w,z)|^{-(n+\frac{1}{2})}$, so using the fact, contained in the proof of Lemma \ref{size and smoothness}, that $d(z,w) \lesssim |g(w,z)|$, we deduce that $|K(z,w)| \lesssim d(z,w)^{-(n+\frac{1}{2})}$. For completeness, we sketch the argument given in \cite{R}.

First, note from \eqref{13} that we can write $N(w,z)=N_0(w,z)+N_1(w,z)$, where $N_0(w,z)= -\left((\bar\partial \eta)^{n-1}) \wedge \bar\partial_w g \wedge \eta \right)$ and $N_1(w,z)=-\left(g(\bar \partial \eta)^n\right).$ Note that $N_0(w,w)=-\left( \partial_w \bar \partial_w \rho \wedge \bar \partial_w \rho \wedge \partial_w \rho\right),$ so in particular $N_0(w,w)$ is a real-valued $(n,n)$ form. Write $N_0(w,z)= \mathcal N_0(w,z) dV(w)$ and $N_1(w,z)= \mathcal{N}_1(w,z) dV(w).$ Moreover, it is clear $\mathcal N_0(w,z)=\mathcal N_0(w,w)+\mathcal{O}(|w-z|)$ by our smoothness assumptions and the same is true of $\mathcal N_0(z,w)$.  Thus, we have, using the fact that $|g(w,z)| \approx |g(z,w)|$ and that the kernel of $\mathcal{T}_2$ is uniformly bounded by a constant $C$:

\begin{eqnarray*}
|K(z,w)| & \lesssim & \left| \frac{ \overline{\mathcal{N}_0(z,w)}}{\overline{g(z,w)}^{n+1}}+\frac{\overline{g(z,w)\mathcal{N}_1(z,w)}}{\overline{g(z,w)}^{n+1}}-\left( \frac{\mathcal{N}_0(w,z)}{g(w,z)^{n+1}}+\frac{g(w,z)\mathcal{N}_1(w,z)}{g(w,z)^{n+1}}\right)\right|+C\\
& \lesssim & \left| \frac{ \overline{\mathcal{N}_0(z,w)}}{\overline{g(z,w)}^{n+1}}- \frac{\mathcal N_0(w,z)}{g(w,z)^{n+1}}\right| +\frac{1}{|g(w,z)|^{n}}\\
& \lesssim & \left| \mathcal{N}_0(w,w)\left(\frac{1}{\overline{g(z,w)}^{n+1}}- \frac{1}{g(w,z)^{n+1}}\right)\right| +\frac{|w-z|}{|g(w,z)|^{n+1}}+\frac{1}{|g(w,z)|^{n}}.
\end{eqnarray*}

Moreover, \cite[Lemma 7.4]{R} gives that $|g(w,z)-\overline{g(z,w)}|= \mathcal{O}(|w-z|^3)$ with an argument very similar to Proposition \ref{diff of Levi}. Then proceeding as in Lemma \ref{kernel difference} and using the fact that $|w-z| \lesssim |g(w,z)|^{1/2}$ yields the desired conclusion.

 The other estimate is proven in the same way, using the fact that $d(z,bD) \lesssim |g(w,z)|$ and $d(w,bD) \lesssim |g(w,z)|$.
\end{proof}

\end{lemma}

We have the following lemma concerning the behavior of $B_1$ weights when integrated against this kernel:

\begin{lemma}\label{B_1 control} Let $\sigma \in B_1$. Then we have the following bounds for all $z,w \in D$ and $\delta>0$:

$$\int_{B(z,\delta)} |K(z,w)|\sigma(w) \mathop{d V(w)} \lesssim (\delta^{1/2}+d(z,bD)^{1/2})\sigma(z)$$ and

$$\int_{B(w,\delta)} |K(z,w)|\sigma(z) \mathop{d V(z)} \lesssim (\delta^{1/2}+d(w,bD)^{1/2})\sigma(z)$$

\begin{proof}

By symmetry, it clearly suffices to prove the first assertion. Let $N$ be the largest non-negative integer so that $2^{-N}\delta>d(z,bD)$. If there is no such $N$, make the obvious modifications. We have, integrating over dyadic ``annuli"

\begin{eqnarray*}
\int_{B(z,\delta)} |K(z,w)| \sigma(w) \mathop{d V(w)} & = & \sum_{j=0}^{\infty} \int_{B(z,2^{-j}\delta)\setminus B(z,2^{-(j+1)}\delta)} |K(z,w)| \sigma(w) \mathop{d V(w)}\\
& = & \sum_{j=0}^{N} \int_{B(z,2^{-j}\delta)\setminus B(z,2^{-(j+1)}\delta)} |K(z,w)| \sigma(w) \mathop{d V(w)}\\
& + & \sum_{j=N+1}^{\infty} \int_{B(z,2^{-j}\delta)\setminus B(z,2^{-(j+1)}\delta)} |K(z,w)| \sigma(w) \mathop{d V(w)}.
\end{eqnarray*}

We deal with the first summation first. We have 

\begin{small}

\begin{eqnarray*}
\sum_{j=0}^{N} \int_{B(z,2^{-j}\delta)\setminus B(z,2^{-(j+1)}\delta)} |K(z,w)| \sigma(w) \mathop{d V(w)} & \lesssim & 
\sum_{j=0}^{N} \int_{B(z,2^{-j}\delta)\setminus B(z,2^{-(j+1)}\delta)} d(z,w)^{-(n+1/2)}\sigma(w) \mathop{d V(w)}\\
& \leq & \sum_{j=0}^{N} \int_{B(z,2^{-j}\delta)} 2^{(j+1)(n+1/2)} \delta^{-(n+1/2)}\sigma(w) \mathop{d V(w)}\\
& \lesssim & \sum_{j=0}^{N} \delta^{1/2} 2^{-j/2} \frac{1}{V(B(z,2^{-j}\delta))}\int_{B(z,2^{-j}\delta)} \sigma(w) \mathop{d V(w)}\\
& \leq & \sum_{j=0}^{N} \delta^{1/2} 2^{-j/2} \mathcal{M}(\sigma)(z)\\
& \lesssim & \delta^{1/2} \mathcal{M}(\sigma)(z)\\
& \lesssim & \delta^{1/2} \sigma(z).
\end{eqnarray*}
\end{small}

Note the implicit constant is independent of $N$. We now proceed to deal with the second summation:

\begin{eqnarray*}
\sum_{j=N+1}^{\infty} \int_{B(z,2^{-j}\delta)\setminus B(z,2^{-(j+1)}\delta)} |K(z,w)| \sigma(w) \mathop{d V(w)} & \leq &  \int_{B(z,d(z,bD))} |K(z,w)| \sigma(w) \mathop{d V(w)}\\
& \lesssim & \int_{B(z,d(z,bD))} d(z,bD)^{-(n+1/2)} \sigma(w) \mathop{d V(w)}\\
& = & \frac{d(z,bD)^{1/2}}{V(B(z,d(z,bD)))} \int_{B(z,d(z,bD))} \sigma(w) \mathop{d V(w)}\\
& \leq & d(z,bD)^{1/2} \mathcal{M}(\sigma)(z)\\
& \lesssim & d(z,bD)^{1/2} \sigma(z).
\end{eqnarray*}
This establishes the result.

\end{proof}
\end{lemma}

Now we will engage in a series of arguments very similar to what is proven in the Szeg\H{o} section. We first note that $\mathcal{T}^*-\mathcal{T}$ improves $L^p$ spaces. The proof of this fact is basically identical to that of Proposition \ref{Lp improvement}  and stems from the fact that $\mathcal{T}^*-\mathcal{T}$ has an ``integrable kernel", so we omit it.

\begin{proposition}\label{Lp improvement2}
The operator $\mathcal{T}^*-\mathcal{T}$ maps $L^p(D)$ to $L^{p+\varepsilon}(D)$ boundedly for $p \geq 1$ and $\varepsilon \in [0, \frac{1}{2n+1}).$
\end{proposition}

The exact same reasoning from Lemma \ref{spectrum} yields the following:

\begin{corollary}\label{spectrum2}
If $\sigma \in B_p$, then $1$ is not an eigenvalue of $\mathcal{T}^*-\mathcal{T}$ considered as an operator on $L^p_\sigma(D)$.
\end{corollary}

It remains to prove that $\mathcal{T}^*-\mathcal{T}$ is compact on $L^p_\sigma(D)$ for $\sigma \in B_p$. The argument is again a reprise of the reasoning in the preceding section, namely Lemma \ref{compactness}.

\begin{lemma}\label{compactness2} The operator $\mathcal{T}^*-\mathcal{T}$ is compact on  on $L^p_\sigma(D)$ for $\sigma \in B_p$.

\begin{proof}

We first note that an integral operator with kernel $K$ bounded on $D \times D$ is automatically compact on $L^p_\sigma(D)$ for $\sigma \in B_p$; the proof follows as in Theorem \ref{szego main detailed}.

To pass to the case where $K$ is unbounded, let $\delta_j=\frac{1}{j}$ and
$$K_j(z,w)=
\begin{cases}
K(z,w) & d(z,w) \geq \delta_j, d(z,bD) \geq \delta_j \text{ or } d(w,bD) \geq \delta_j\\
0 & \text{otherwise}
\end{cases}.$$

Let $\mathcal{T}_j$ be the integral operator with kernel $K_j$. Note that $K_j$ is bounded on $D \times D$ because $|K(z,w)| \lesssim \frac{1}{|g(w,z)|^{n+1/2}}$ and $|g(w,z)| \gtrsim |\rho(w)|+|\rho(z)|+|z-w|^2$ by \cite[Proposition 2.1]{LS2}. Thus $\mathcal{T}_j$ is compact on $L^p_\sigma(D)$. To show $\mathcal{T}$ is compact, it suffices to show $\mathcal{T}_j \rightarrow \mathcal{T}$ in operator norm.

\vspace{0.25 cm}

To this end, let $f \in L^p_\sigma(D)$ with $||f||_{L^p_\sigma(D)}\leq 1.$ Note that as $\sigma \in B_p$, we can write $$\sigma=\frac{\sigma_1}{\sigma_2^{p-1}}$$ where $\sigma_1,\sigma_2 \in B_1$ by the factorization of $B_p$ weights. This factorization of $B_p$ weights holds by the arguments in \cite{Ru}; note that the adapted maximal function $\mathcal{M}$ is bounded on $L^p_\sigma(D)$ for $\sigma \in B_p$ and if $\sigma \in B_p$, $\sigma^{-\frac{1}{p-1}} \in B_q$ where $p^{-1}+q^{-1}=1$, so \cite[Theorem 2]{Ru} can be applied. It should also be noted that this factorization appears in the literature in the context of the unit disk $\mathbb{D}$; see \cite{Bori}. By H\"{o}lder's Inequality applied to the functions $$|K(z,w)-K_j(z,w)|^{1/q}\sigma_2(w)^{1/q} \quad \text{and} \quad |K(z,w)-K_j(z,w)|^{1/p}\sigma_2(w)^{-1/q}|f(w)|$$ and then applying Proposition \ref{B_1 control}, we obtain the estimate:

\begin{small}
\begin{align*}
& |(\mathcal{T}-\mathcal{T}_j)(f)(z)|\\
& \leq  \int_{bD}|K(z,w)-K_j(z,w)||f(w)| \mathop{d V(w)}\\
 & = \chi_{d(z,bD) < \delta_j}\left(\int_{B(z, \delta_j)} |K(z,w)| \sigma_2(w) \mathop{d V(w)}\right)^{\frac{1}{q}} \left(\int_{B(z,\delta_j) \cap\{d(w,bD)<\delta_j\}} |K(z,w)| (\sigma_2(w))^{1-p}|f(w)|^p \mathop{d V(w)} \right)^{\frac{1}{p}}\\
& \lesssim \delta_j^{1/2q} \sigma_2(z)^{\frac{1}{q}} \left(\int_{B(z,\delta_j)\cap\{d(w,bD)<\delta_j\}}  |K(z,w)| (\sigma_2(w))^{1-p}|f(w)|^p \mathop{d V(w)} \right)^{\frac{1}{p}}.\\
\end{align*}
\end{small}

Thus, we obtain, applying the proceeding estimate, Fubini, and Proposition \ref{B_1 control} again: 

\begin{align*}
& ||(T-T_j)f||^p_{L^p_\sigma(bD)}\\
 & \leq \int_{D} \delta_j^{\frac{p}{2q}} \sigma_2(z)^{p-1} \left(\int_{B(z,\delta_j)\cap\{d(w,bD)<\delta_j\}} |K(z,w)| (\sigma_2(w))^{1-p} |f(w)|^p \mathop{d V(w)} \right)\frac{\sigma_1(z)}{\sigma_2(z)^{p-1}} \mathop{d V(z)}\\
 & =  \delta_j^{\frac{p}{2q}} \int_{D} \int_{B(z,\delta_j)\cap\{d(w,bD)<\delta_j\}} |K(z,w)| (\sigma_2(w))^{1-p} |f(w)|^p \mathop{d V(w)} \sigma_1(z)\mathop{d V(z)}\\
& =  \delta_j^{\frac{p}{2q}} \int_{D} \chi_{d(w,bD)<\delta_j}(w) \left(\int_{B(w,\delta_j)} |K(z,w)| \sigma_1(z)\mathop{d V(z)}\right) |f(w)|^p  (\sigma_2(w))^{1-p}\mathop{d V(w)}\\
& \lesssim  \delta_j^{p/2} \int_{D} \sigma_1(w) |f(w)|^p  (\sigma_2(w))^{1-p}\mathop{d V(w)}\\
 & =  \delta_j^{p/2} ||f||^p_{L^p_\sigma(D)}\\
&  \leq  \delta_j^{p/2}.
\end{align*}

Letting $j \rightarrow \infty$, we have $\delta_j \rightarrow 0$ and thus it immediately follows that the operators $\mathcal{T}_j$ approach $\mathcal{T}$ in operator norm and hence $\mathcal{T}$ is compact.

\end{proof}

\end{lemma}

\subsection{Proof of Main Theorem}

We now can finally prove Theorem \ref{bergman main detailed}, using the Kerzman-Stein operator equation trick.

\begin{proof}[Proof of Theorem \ref{bergman main detailed}]
The proof is virtually identical to that of Theorem \ref{szego main detailed}. Again, the starting point is the Kerzman-Stein equation, and the invertibility of $(I-(\mathcal{T}^*-\mathcal{T}))$ on $L^p_\sigma(D)$ is granted by Corollary \ref{spectrum2} and Lemma \ref{compactness2} using the spectral theorem. The boundedness of $\mathcal{T}$ on $L^p_\sigma(D)$ is given by Theorem \ref{aux operator}.

\end{proof}

\begin{section}{Acknowledgements}
The authors would like to thank Cody Stockdale for helpful discussions related to this work.
\end{section}

\end{document}